\documentclass[final,3p,times]{elsarticle}

\usepackage{amsmath}
\usepackage{amssymb}
\usepackage{amsthm}
\usepackage{enumerate}
\usepackage{epsfig}
\usepackage{epstopdf}
\usepackage{graphicx}
\usepackage{amstext}
\usepackage{textcomp}
\usepackage{amsfonts}
\usepackage{caption}
\usepackage{color}
\usepackage[normalem]{ulem}

\newtheorem{theorem}{Theorem}
\newtheorem{lemma}{Lemma}
\newtheorem{definition}{Definition}
\newtheorem{assumption}{Assumption}
\newproof{pt}{Proof}
\newproof{pot}{Proof of Theorem \ref{th:main}}
\newproof{pot2}{Proof of Theorem \ref{th:3456}}
\newproof{sol}{Solution}

\journal{https://arxiv.org/archive/math}

\begin{document}

\begin{frontmatter}



\title{Application of Coupled Fixed (or Best Proximity) Points in Market Equilibrium in Oligopoly Markets.}


\author[uni,dzh]{Y. Dzhabarova}
\author[uni,kab]{S. Kabaivanov}
\author[uni,rus]{M. Ruseva}
\author[uni,zla]{B. Zlatanov}

\address[uni]{University of Plovdiv ``Paisii Hilendarski'',
24 ``Tzar Assen'' Str.,
Plovdiv, 4000,
Bulgaria}
\fntext[dzh]{j\_jabarova@yahoo.com}
\fntext[kab]{stanimir.kabaivanov@gmail.com}
\fntext[rus]{ruseva\_margarita@mail.bg}
\fntext[zla]{bzlatanov@gmail.com}

\begin{abstract}
We present a possible kind of generalization of the notion of ordered pairs of cyclic maps and coupled fixed points and its application in modelling of equilibrium in oligopoly markets. We have obtained sufficient conditions for the existence and uniqueness of fixed (or best proximity) points in complete metric spaces (uniformly convex Banach spaces). We get an error estimates of the fixed (or best proximity), provided that we have used sequences of successive iterations.
We illustrate one possible application of the results by building a pragmatic model on competition in oligopoly markets. To achieve this goal, we use an approach based on studying the response functions of each market participant, thus making it possible to address both Cournot and Bertrand industrial structures with unified formal method. In contrast to the restrictive theoretical constructs of duopoly equilibrium, our study is able to account for real-world limitations like minimal sustainable production levels and exclusive access to certain resources. We prove and demonstrate that by using carefully constructed response functions it is possible to build and calibrate a model that reflects different competitive strategies used in extremely concentrated markets. The response functions approach makes it also possible to take into consideration different barriers to entry. By fitting to the response functions rather than the profit maximization of the payoff functions problem we alter the classical optimization problem to a problem of coupled fixed points, which has the benefit that considering corner optimum, corner equilibria and convexity condition of the payoff function can be skipped.
\end{abstract}

\begin{keyword}

fixed point \sep coupled fixed points \sep duopoly equilibrium \sep response functions \sep imperfect competition
\MSC[2010] 41A25 \sep 47H10 \sep 54H25 \sep 46B20
\end{keyword}

\end{frontmatter}



\section{Introduction}

The Banach contraction principle is a powerful tool both in pure and applied mathematics. 
It states that in a complete metric space $(X,\rho)$ any contraction map $T:X\to X$ has a fixed point, i.e. $\min\{\rho(x,Tx):x\in X\}=0$. A lot of results in modelling real world processes in applied mathematics lead to the problem of finding $\min\{\rho(x,Tx):x\in X\}$. It may happen that the above minimum is greater than zero. One approach for tackling the over said issues utilizes the idea of best proximity points, firstly presented in \cite{EV}, where a sufficient condition for the existence and the uniqueness of best proximity points in uniformly convex Banach spaces is obtained.

There is a great number of 
generalizations of the mentioned above principle. One class of such generalizations appears, when the constructed model may depends on two parameters, 
i.e. $F:X\times X\to X$. The notion of coupled fixed points \cite{GL} and of a coupled best proximity points for an ordered pair $(F,G)$, $F:A\times A\to B$, $G:B\times B\to A$, where $A,B\subset X$  \cite{GRK,SK}, is relevant in this context. Deep results in the theory of coupled fixed points can be found for example in \cite{Berinde22,Berinde23,BP}. 

Following \cite{Cournot,Friedman,Smith} the model of duopoly market consists of two players $i=1,2$, each player corrects its production accordingly to his production and the productions of the other player, i.e. $F_i:A_1\times A_2\to A_i$, where $F_i$ be the response function of the $i$ player and $A_i$ be the production set of the $i$.
That is why we  could not apply the known results about coupled fixed (or best proximity) points results, i.e considering the response functions to be $F:A\times A\to B$, $G:B\times B\to A$, in the theory of oligopoly (duopoly) markets. 

There are numerous issues about fixed points and best proximity points that are not simple to be solved or can not be solved precisely. Well known benefits of fixed points results are the error estimates of the successive iterations and the rate of convergence. That is why an estimation of the error when an iterative process is used is of interest, when fixed points or best proximity points are investigated. An extensive study about approximations of fixed points can be found in \cite{Ber}. A to begin with result within the approximation of the sequence of successive iterations, converging to the best proximity point for cyclic contractions, is obtained in \cite{Z}.  This result was extended about a coupled best proximity point in \cite{Ilchev-zlatanov,Zlatanov-FPT}. 

We have illustrated in the application section that starting form any initial level of production an equilibrium point exist and we have calculated the error estimates.

\section{Preliminaries}


We will recall the needed notions and results, that we will use.

Let $(X,\rho )$ be a metric space.
A distance between two subsets $A,B\subset X$ is defined by ${\rm dist}(A,B)=\inf\{\rho (x,y):x\in A, y\in B\}$.
Following \cite{EV} let $A$ and $B$ be nonempty subsets of a metric space $(X,\rho )$.
The map $T:A\bigcup B\to A\bigcup B$ is called a cyclic map if $T(A)\subseteq B$ and $T(B)\subseteq A$.
A point $\xi\in A$ is called a best proximity point of the cyclic map $T$ in $A$ if $\rho (\xi,T\xi )={\rm dist}(A,B)$.

\begin{definition}\label{defSK1}(\cite{SK})
	Let $A$ and $B$ be nonempty subsets of a metric space $(X,\rho)$, $F:A\times A\to B$. An ordered pair $(x,y)\in A\times A$ is called a coupled best proximity point of $F$ if
	$\rho(x,F(x,y))=\rho(y,F(y,x))={\rm dist}(A,B)$.
\end{definition}

Let $A$ be nonempty subset of a metric space $(X,\rho )$.
The map $T:A\to A$ is said to have a fixed point $x\in A$ if $\rho (\xi,T\xi )=0$.

\begin{definition}\label{defGL}(\cite{GL})
	Let $A$ and $B$ be nonempty subsets of a metric space $(X,\rho)$, $F:A\times A\to A$. An ordered pair $(x,y)\in A\times A$ is said to be a coupled fixed point of $F$ in $A$ if $x=F(x,y)$ and $y=F(y,x)$.
\end{definition}

In order to apply the technique of coupled best proximity points and coupled fixed points we will generalize the mentioned above notions.
When we investigate duopoly with players' response functions $F$ and $f$, we have seen that each player using the information about his production and the 
rival's production choose a change in his production, i.e. we define $F:A\times B\to A$ instead of the cyclic type of maps $F:A\times B\to B$ (Definition \ref{defSK1}).
Therefore  we introduce generalizations of Definition \ref{defSK1} and Definition \ref{defGL}.

\begin{definition}\label{new-fixed-point}
	Let $A_x$, $A_y$ be nonempty subsets of a metric space $(X,\rho)$, $F:A_x\times A_y\to A_x$, $f:A_x\times A_y\to A_y$. 
	An ordered pair $(\xi,\eta)\in A_x\times A_y$ is called a coupled fixed point of $(F,f)$ if
	$\xi=F(\xi,\eta)$ and $\eta=f(\xi,\eta)$.
\end{definition}

\begin{definition}\label{new-best-proximity}
	Let $A_x$, $A_y$ be nonempty subsets of a metric space $(X,\rho)$, $F:A_x\times A_y\to A_x$, $f:A_x\times A_y\to A_y$. 
	An ordered pair $(\xi,\eta)\in A_x\times A_y$ is called a coupled best proximity point of $(F,f)$ if
	$\rho(\eta,F(\xi,\eta)=\rho(\xi,f(\xi,\eta)={\rm dist}(A_x,A_y)$.
\end{definition}

\begin{definition}\label{iterated_sequence}
	Let $A_x$, $A_y$ be nonempty subsets of $X$. Let $F:A_x\times A_y\to A_x$, $f:A_x\times A_y\to A_y$.
	For any pair $(x,y)\in A_x\times A_y$ we define the sequences $\{x_n\}_{n=0}^\infty$ and $\{y_n\}_{n=0}^\infty$ by
	$x_0=x$, $y_0=y$ and $x_{n+1}=F(x_{n},y_{n})$, $y_{n+1}=f(x_{n},y_{n})$ for all $n\geq 0$.
\end{definition}

Everywhere, when considering the sequences $\{x_n\}_{n=0}^\infty$ and $\{y_n\}_{n=0}^\infty$ we will assume that they are the sequences defined in Definition \ref{iterated_sequence}.

We will generalized the contraction condition from \cite{EV} for the maps, defined in Definition \ref{new-fixed-point} and Definition \ref{new-best-proximity}.

\begin{definition}\label{cyclic-contraction}
	Let $A_x$, $A_y$ be nonempty subsets of a metric space $(X,\rho)$. Let there exist a subset $D\subseteq A_x\times A_y$ and maps $F:D\to A_x$ and $f:D\to A_y$, such that $(F(x,y),f(x,y))\subseteq D$ for every $(x,y)\in D$. 
	The ordered pair of ordered pairs $(F,f)$ is said to be a cyclic contraction of type one ordered pair
	if there exist non-negative numbers $\alpha,\beta$, such that $\max\{\alpha+\gamma,\beta+\delta\}<1$ and there holds the inequality
	\begin{equation}\label{eq-cyclic-contraction}
	\rho(F(x,y),F(u,v))+\rho(f(z,w),f(t,s)\leq \alpha \rho(x, u)+\beta \rho(y,v)+\gamma\rho(z, t)+\delta\rho(w,s)
	\end{equation}
	for all $(x,y), (u,v), (z,w),(t,s)\in D$.
\end{definition}

\begin{definition}\label{cyclic-contraction-2}
	Let $A_x$, $A_y$ be nonempty subsets of a metric space $(X,\rho)$. Let there exist a subset $D\subseteq A_x\times A_y$ and maps $F:D\to A_x$ and $f:D\to A_y$, such that $(F(x,y),f(x,y))\subseteq D$ for every $(x,y)\in D$. 
	The ordered pair of ordered pairs $(F,f)$ is said to be a cyclic contraction of type two ordered pair
	if there exist non-negative numbers $\alpha,\beta$, such that $\alpha +\beta<1$ and there holds the inequality
	\begin{equation}\label{eq-cyclic-contraction-1}
	\rho(F(x,y),f(u,v))\leq \alpha \rho(x, v)+\beta \rho(y,u)+(1-(\alpha +\beta)){\rm dist}(A_x,A_y)
	\end{equation}
	for all $(x,y), (u,v)\in D$.
\end{definition}

The norm--structure of the underlying space plays a crucial role in the proofs \cite{EV}.

Whenever we deal with a distance in $(X,\|\cdot\| )$, we will always assume that it is generated by the norm $\|\cdot\|$ i.e. $\rho (x,y)=\|x-y\|$.

The uniformly convexity plays a vital part within the proofs of best proximity points.

\begin{definition}
	Let $(X,\|\cdot\|)$ be a Banach space. For every $\varepsilon\in (0,2]$ we define the modulus of convexity of $\|\cdot\|$ by
	$$
	\delta_{\|\cdot\|}(\varepsilon)=\inf\left\{1-\left\|\frac{x+y}{2}\right\|:\|x\|\leq 1,\|y\|\leq 1, \|x-y\|\geq\varepsilon\right\}.
	$$
	The norm is called uniformly convex if $\delta_X(\varepsilon)>0$ for all $\varepsilon\in (0,2]$. The space $(X,\|\cdot\|)$ is then called a uniformly convex space.
\end{definition}

\begin{lemma}(\cite{EV})\label{Eldred-2}
	Let $A$ be a nonempty closed, convex subset, and $B$ be a nonempty closed subset of a uniformly convex Banach space. Let
	$\{x_n\}_{n=1}^\infty$ and $\{z_n\}_{n=1}^\infty$ be sequences in $A$ and $\{y_n\}_{n=1}^\infty$ be a sequence in $B$ satisfying:\\
	1) $\lim_{n\to\infty}\|x_n-y_n\|={\rm dist} (A,B)$;\\
	2) $\lim_{n\to\infty}\|z_n-y_n\|={\rm dist} (A,B)$;\\
	then $\lim_{n\to\infty}\|x_n-z_n\|=0$.
\end{lemma}
\begin{lemma}(\cite{EV})\label{Eldred-1}
	Let $A$ be a nonempty closed, convex subset, and $B$ be a nonempty closed subset of a uniformly convex Banach space. Let
	$\{x_n\}_{n=1}^\infty$ and $\{z_n\}_{n=1}^\infty$ be sequences in $A$ and $\{y_n\}_{n=1}^\infty$ be a sequence in $B$ satisfying:\\
	1) $\lim_{n\to\infty}\|z_n-y_n\|={\rm dist} (A,B)$;\\
	2) for every $\varepsilon >0$ there exists $N_0\in\mathbb{N}$, such that for all $m>n\geq N_0$, $\|x_n-y_n\|\leq {\rm dist} (A,B)+\varepsilon$,\\
	then for every $\varepsilon >0$, there exists $N_1\in\mathbb{N}$, such that for all $m>n>N_1$, holds $\|x_m-z_n\|\leq\varepsilon$.
\end{lemma}

The inequality
\begin{equation}\label{eq:1}
\left\|\frac{x+y}{2}-z\right\|\leq \left(1-\delta_X\left(\frac{r}{R}\right)\right)R
\end{equation}
holds for any $x,y,z\in X$, $R>0$, $r\in [0,2R]$, $\|x-z\|\leq R$, $\|y-z\|\leq R$ and $\|x-y\|\geq r$, provided that 
$X$ is a uniformly convex \cite{EV}.

The modulus of convexity $\delta_X(\varepsilon)$ is a strictly increasing function, provided that the underlaying space is uniformly convex, and its inverse function $\delta^{-1}$ exists.
If the inequality $\delta_{\|\cdot\|}(\varepsilon)\geq C\varepsilon^q$ holds for some constants $C,q>0$ and for any $\varepsilon\in (0,2]$, the modulus of convexity is said to be of power type $q$.
The moduli of convexity with respect to the $p$--norm in $\ell_p$ or $L_p$ are of power type and the inequalities $\delta_{\|\cdot\|_p}(\varepsilon)\geq \frac{\varepsilon^p}{p2^p}$ for $p\geq 2$ and
$\delta_{\|\cdot\|_p}(\varepsilon)\geq \frac{(p-1)\varepsilon^2}{8}$ for $p\in (1,2)$ hold \cite{M}.

A comprehensive presenting of the results from this section can be found in \cite{BB, DGZ, FHHMPZ}.

\section{Main Results}

\subsection{Coupled fixed points}

\begin{theorem}\label{th:3456}
	Let $A_x$, $A_y$ be nonempty and closed subsets of a complete metric space $(X,\rho)$. Let there exist a closed subset $D\subseteq A_x\times A_y$ and maps $F:D\to A_x$ and $f:D\to A_y$, such that $(F(x,y),f(x,y))\subseteq D$ for every $(x,y)\in D$. 
	Let the ordered pair $(F,f)$ be a cyclic contraction of type one.
	Then
	\begin{enumerate}[\normalfont (I)]
		\item\label{item_I-3} There exists a unique pair $(\xi,\eta)$ in $D$, which is a unique coupled fixed point for the ordered pair $(F,f)$.
		Moreover the iteration sequences $\{x_{n}\}_{n=0}^\infty$ and $\{y_n\}_{n=0}^\infty$, defined in Definition \ref{iterated_sequence}
		converge to $\xi$ and $\eta$ respectively, for any arbitrary chosen initial guess $(x,y)\in A_x\times A_y$;
		\item\label{item_II-3} a priori error estimates hold
		$\max\left\{\rho(x_{n},\xi),\rho(y_{n},\eta)\right\}\leq \frac{k^{n}}{1-k}(\rho(x_1,x_0)+\rho(y_1,y_0))$;
		\item\label{item_III-prime-3} a posteriori error estimates hold
		$\max\left\{\rho(x_{n},\xi),\rho(y_{n},\eta)\right\}\leq\frac{k}{1-k}(\rho(x_{n-1},x_{n})+\rho(y_{n-1},y_{n}))$;
		\item\label{item_IV} rate of convergence for the sequences of successive iterations 		
		$\rho (x_n,\xi)+\rho (y_n,\eta)\leq k\left(\rho (x_{n-1},\xi)+(y_{n-1},\eta)\right)$,
		where $k=\max\{\alpha+\gamma,\beta+\delta\}$.
	\end{enumerate}
\end{theorem}

{\it Proof:} Let us choose an arbitrary point $(x,y)\in D$ and $\{x_{n}\}_{n=0}^\infty$, $\{y_n\}_{n=0}^\infty$ be the sequences defined in Definition \ref{iterated_sequence}.	
Then for any $n\in \mathbb{N}$ there hold the chain of inequalities
$$
\begin{array}{lll}
\rho(x_{n+1},x_{n})+\rho(y_{n+1},y_{n})&=&\rho(F(x_{n},y_{n}),F(x_{n-1},y_{n-1})+\rho(f(x_{n},y_{n}),f(x_{n-1},y_{n-1})\\
&\leq& \alpha\rho(x_{n},x_{n-1})+\beta\rho(y_{n},y_{n-1})+\gamma\rho(x_{n},x_{n-1})+\delta\rho(y_{n},y_{n-1})\\
&=&(\alpha+\gamma)\rho(x_{n},x_{n-1})+(\beta+\delta)\rho(y_{n},y_{n-1})\leq \max\{\alpha+\gamma,\beta+\delta\}(\rho(x_{n},x_{n-1})+\rho(y_{n},y_{n-l})).
\end{array}
$$
Simply to fit a few of the equations within the content field we will denote $k=\max\{\alpha+\gamma,\beta+\delta\}$.
Consequently
\begin{equation}\label{eq:567}
\rho(x_{n+1},x_{n})+\rho(y_{n+1},y_{n})\leq k^l(\rho(x_{n+1-l},x_{n-l})+\rho(y_{n+1-l},y_{n-l})).
\end{equation}
(\ref{item_I-3}) From (\ref{eq:567}), applied for $l=n$ we get
$$
\max\left\{\rho(x_{n+1},x_{n}),\rho(y_{n+1},y_{n})\right\}\leq k^{n}(\rho(x_1,x_0)+\rho(y_1,y_0)).
$$
Thus
\begin{equation}\label{eq:19}
\rho(x_{n},x_{n+m})\leq\displaystyle\sum_{j=n}^{n+m-1}\rho(x_{j},x_{j+1})\leq \sum_{j=n}^{n+m-1}k^{j}(\rho(x_1,x_0)+\rho(y_1,y_0))\leq \displaystyle k^{n}\frac{1-k^{m}}{1-k}(\rho(x_1,x_0)+\rho(y_1,y_0)).
\end{equation}
Since $k\in (0,1)$ it follows that \sout{the sequence} $\{x_n\}_{n=0}^\infty$ is a Cauchy sequence in $A_x$.
Thus $\{x_{n}\}$ converges to some
$\xi$.

The verification that $\{y_n\}_{n=0}^\infty$ converges to some $\eta \in A_y$ can be completed in a similar mold.
From the assumption that $D$ is closed it follows that $(\xi,\eta)\in D$.

We will prove that the pair $(\xi,\eta)$ is a coupled fixed point of $(F,f)$.  
By the triangle inequality and (\ref{eq-cyclic-contraction}) we get the inequalities
$$
\begin{array}{lll}
\rho (\xi,F(\xi,\eta))+\rho (\eta,f(\xi,\eta))&\leq& \rho (\xi,x_{n})+\rho(x_{n},F(\xi,\eta))+\rho (\eta,y_{n})+\rho(y_{n},f(\xi,\eta))\\
&\leq&
\rho (\xi,x_{n})+\rho(F(x_{n-1},y_{n-1}),F(\xi,\eta))+\rho (\eta,y_{n})+\rho(f(x_{n-1},y_{n-1}),f(\xi,\eta))\\
&\leq&\rho (\xi,x_{n})+\alpha\rho(x_{n-1},\xi)+\beta\rho(y_{n-1},\eta)+\rho (\eta,x_{n})+\gamma\rho(x_{n-1},\xi)+\delta\rho(y_{n-1},\eta).
\end{array}
$$
Taking a limit when $n\to\infty$, we get $\rho (\xi,F(\xi,\eta))+\rho (\eta,F(\eta,\xi))=0$, i.e.
$\rho (\xi,F(\xi,\eta))=0$ and $\rho (\eta,F(\eta,\xi))=0$. Consequently $(\xi,\eta)$ is a coupled fixed point of $(F,f)$.

We will prove that $(\xi,\eta)$ is unique. Let us assume the contrary, i.e. there is
$(\xi^{*},\eta^{*})\in D\subseteq A_x\times A_y$ so that $(\xi^{*},\eta^{*})\not =(\xi,\eta)$ and
$\xi^{*}=F(\xi^{*},\eta^{*})$, $\eta^{*}=f(\xi^{*},\eta^{*})$. The inequalities
$$
\begin{array}{lll}
\rho (\xi^{*},\xi)+\rho (\eta^{*},\eta)&=&\rho (F(\xi^{*},\eta^{*}),F(\xi,\eta)+\rho (f(\eta^{*},\xi^{*}),f(\eta,\xi)
\leq
\alpha\rho(\xi^{*},\xi)+\beta\rho(\eta^{*},\eta)+\gamma\rho(\xi^{*},\xi)+\delta\rho(\eta^{*},\eta)\\
&=&(\alpha+\gamma)\rho(\xi^{*},\xi)+(\beta+\delta)\rho(\eta^{*},\eta)<\rho(\xi^{*},\xi)+\rho(\eta^{*},\eta)
\end{array}
$$
result to
$\rho (\xi^{*},\xi)=\rho (\eta^{*},\eta)=0$, a contradiction and consequently the  coupled fixed point $(\xi,\eta)$ of $(F,f)$ is unique .

(\ref{item_II-prime-2}) Letting $m\to\infty$ in (\ref{eq:19}) we get the a priori
estimate
$\rho(x_{n},\xi)\leq \frac{k^{n}}{1-k}(\rho(x_1,x_0)+\rho(y_1,y_0))$.
The proof that $\rho (y_{n},\eta)\leq\frac{k^{n}}{1-k}(\rho(x_1,x_0)+\rho(y_1,y_0))$
is completed by similar arguments. Therefore
$$
\max\{\rho(x_{n},\xi),\rho (y_{n},\eta)\}\leq \frac{k^{n}}{1-k}(\rho(x_1,x_0)+\rho(y_1,y_0)).
$$

(\ref{item_III-prime-3})
By (\ref{eq:567}) applied  for $l=j+1$ we get
$$
\rho (x_{n},x_{n+m})\leq\displaystyle\sum_{j=0}^{m-1}\rho(x_{n+j},x_{n+j+1})\leq\sum_{j=0}^{m-1}k^{j+1}(\rho(x_{n-1},x_{n})+\rho(y_{n-1},y_{n}))\leq\displaystyle\frac{k}{1-k}(1-k^{m+1})(\rho(x_{n-1},x_{n})+\rho(y_{n-1},y_{n})).
$$
Letting $m\to\infty$ we get the a posteriori estimate
$\rho (x_{n},\xi)\leq\frac{k}{1-k}(\rho(x_{n-1},x_{n})+\rho(y_{n-1},y_{n}))$.
The proof that $\rho (y_{n},\eta)\leq\frac{k}{1-k}(\rho(x_{n-1},x_{n})+\rho(y_{n-1},y_{n}))$
is done in a similar fashion and thus
$$
\max\{\rho (x_{n},\xi),\rho (y_{n},\eta)\}\leq\frac{k}{1-k}(\rho(x_{n-1},x_{n})+\rho(y_{n-1},y_{n})).
$$

\ref{item_IV}) Considering that the pair $(\xi,\eta)$ is a coupled fixed point for $(F,f)$ and (\ref{eq-cyclic-contraction}) we have the inequalities
$$
\begin{array}{lll}
\rho (x_{n},\xi)+\rho (y_{n},\eta)&=&\rho(F(x_{n-1},y_{n-1}),F(\xi,\eta))+\rho(f(x_{n-1},y_{n-1}),f(\xi,\eta))\\
&\leq& \alpha\rho (x_{n-1},\xi)+\beta\rho(y_{n-1},\eta)+\gamma\rho (x_{n-1},\xi)+\delta\rho(y_{n-1},\eta)\\
&=&(\alpha+\gamma)\rho (x_{n-1},\xi)+(\beta+\delta)\rho(y_{n-1},\eta)\leq k(\rho (x_{n-1},\xi)+\rho(y_{n-1},\eta)).
\end{array}
$$
Consequently
$\rho (x_{n},\xi)+\rho (y_{n},\eta)\leq k(\rho (x_{n-1},\xi)+\rho(y_{n-1},\eta))$.
\hfill$\Box$

\subsection{Coupled best proximity points}

Simply to fit a few of the equations within the content field let us denote $d={\rm dist}(A_x,A_y)$,
$P_{n,m}(x,y)=\|x_{n}-y_{m}\|$ and $W_{n,m}(x,y)=P_{n,m}(x,y)-d=\|x_{n}-y_{m}\|-d$,
where $x=\{x_n\}_{n=0}^\infty$ and $y=\{y_n\}_{n=0}^\infty$.

\begin{lemma}\label{lemma-1}
	Let $A_x$, $A_y$ be nonempty subsets of a metric space $(X,\rho)$. Let there exist a subset $D\subseteq A_x\times A_y$ and maps $F:D\to A_x$ and $f:D\to A_y$, such that $(F(x,y),f(x,y))\subseteq D$ for every $(x,y)\in D$. 
	Let the ordered pair $(F,f)$ be a cyclic contraction of type two. Then there holds
	$\lim_{n\to\infty}\rho(x_{n},y_{n+k})=d$ and $\lim_{n\to\infty}\rho(x_{n+k},y_{n})=d$ for an arbitrary chosen $(x,y)\in D$ and arbitrary $k=0,1,2,\dots$.	
\end{lemma}

{\it Proof:}
Let us choose an arbitrary $(x,y)\in D$ and define $\left\{x_n\right\}_{n=0}^\infty$, $\left\{y_n\right\}_{n=0}^\infty$   

Using the cyclic contraction condition (\ref{eq-cyclic-contraction-1}) we get that for all $n,k\in\mathbb{N}$ holds
$$
\begin{array}{lll}
\rho(x_{n+1},y_{n+1+k})&=&
\rho\left(F(x_{n},y_{n+k}),f(x_{n+k},y_{n+k})\right)\leq
\alpha\rho(x_{n},y_{n+k})+\beta\rho(y_{n+k},x_{n})+(1-(\alpha+\beta))d\\
&=&(\alpha+\beta)\rho(x_{n},y_{n+k})+(1-(\alpha+\beta))d
\end{array}
$$
Thus we get
\begin{equation}\label{lemma-1-eq-1}
\begin{array}{lll}
\rho(x_{n+1},y_{n+1+k})-d&\leq&
(\alpha+\beta)(\rho(x_{n},y_{n+k})-d)\leq	(\alpha+\beta)^2(\rho(x_{n-1},y_{n-1+k})-d)\\
&\leq&
(\alpha+\beta)^3(\rho(x_{n-2},y_{n-1+k}))-d)\leq\cdots\leq
(\alpha+\beta)^{n+1}(\rho(x_{0},y_{k})-d)).
\end{array}
\end{equation}
For any arbitrary  and fixed $k\in\mathbb{N}$, after taking limit in (\ref{lemma-1-eq-1}), when $n\to\infty$, by using the assumption that $\alpha+\beta\in (0,1)$, we get 
$\lim_{n\to\infty}(\rho(x_{n+1},y_{n+1+k})-d)=0$ and thus we obtain
$\lim_{n\to\infty}\rho(x_{n+1},y_{n+1+k})=d$.

The proof of $\lim_{n\to\infty}\rho(x_{n+k},y_{n})=d$ can be done in a similar fashion.
\hfill$\Box$

It can be seen easily that (\ref{lemma-1-eq-1}) holds for indexes $m>n$, too.
\begin{equation}\label{lemma-1-eq-2}
\rho(x_{n},y_{m})-d\leq
(\alpha+\beta)^{n}(\rho(x_{0},y_{m-n})-d).\\
\end{equation}

\begin{lemma}\label{lemma-2}
	Let $A_x$, $A_y$ be nonempty subsets of a metric space $(X,\rho)$. Let there exist a subset $D\subseteq A_x\times A_y$ and maps $F:D\to A_x$ and $f:D\to A_y$, so that $(F(x,y),f(x,y))\subseteq D$ for every $(x,y)\in D$.
	Let the ordered pair $(F,f)$ be a cyclic contraction of type two. 
	The iterrative sequences $\{x_{n}\}_{n=0}^\infty$ and $\{y_{n}\}_{n=0}^\infty$, for any initial guess $(x,y)\in D$ are bounded.
\end{lemma}

{\it Proof:}
Let $(x,y)\in D$ be arbitrary chosen and fixed. From Lemma \ref{lemma-1} we have that  $\lim_{n\to\infty}\rho(x_{n},y_{n})=d$ and thus it will be sufficient to demonstra that only $\{x_{n}\}_{n=0}^\infty$ is a bounded sequence.

Let as choose $$M>\frac{(1-(\alpha+\beta)^2)d+(\alpha+\beta)^2(\rho(y_0,x_2)+\rho(x_2,y_2))}{1-(\alpha+\beta)^2}.$$
Suppose the contrary, i.e.  $\{x_{n}\}_{n=0}^\infty$ is not bounded. Then there exists $n_0\in\mathbb{N}$, such that
there hold $\rho(y_2,x_{n})\leq M$ for all $n<n_0$ and 
\begin{equation}\label{lemma-2-eq-1}\rho(y_2,x_{n_0})> M.\end{equation}
From inequality (\ref{lemma-2-eq-1}) after a substitution in (\ref{lemma-1-eq-2}) with $n=2$ and $m=n_0$ we get
$$
\begin{array}{lll}
\displaystyle\frac{M-d}{(\alpha+\beta)^2}&<&\displaystyle\frac{\rho(y_2,x_{n_0})-d}{(\alpha+\beta)^2}
\leq\rho(y_0,x_{n_0-2})-d
\leq \rho(y_0,x_2)+\rho(x_2,y_2)+\rho(y_2,x_{n_0-2})-d\\
&\leq& \rho(y_0,x_2)+\rho(x_2,y_2)+M-d,
\end{array}
$$
which inequality can hold true only if the inequality
$
M\leq \displaystyle\frac{(1-(\alpha+\beta)^2)d+(\alpha+\beta)^2(\rho(y_0,x_2)+\rho(x_2,y_2))}{1-(\alpha+\beta)^2}$ 
holds, which contradicts with the choice of $M$.
\hfill$\Box$

\begin{lemma}\label{lemma-3}
	Let $A_x$, $A_y$ be nonempty convex subsets of a uniformly convex Banach space $(X,\|\cdot\|)$. Let there exist a subset $D\subseteq A_x\times A_y$ and maps $F:D\to A_x$ and $f:D\to A_y$, such that $(F(x,y),f(x,y))\subseteq D$ for every $(x,y)\in D$.
	Let the ordered pair $(F,f)$ be a cyclic contraction of type two. 
	For any arbitrary chosen $(x,y)\in D$ and for every $\varepsilon>0$ there is $n_0\in\mathbb{N}$ so that the inequality
	$	\|x_m-y_{n}\|<d+\varepsilon$ 	holds for any $m\geq n>n_0$.
\end{lemma}

{\it Proof:}
From Lemma \ref{lemma-1} we get $\lim_{n\to\infty}\|x_{n}-y_{n}\|=d$ and $\lim_{n\to\infty}\|x_{n+1}-y_{n}\|=d$.

By Lemma \ref{Eldred-2} after using the uniform convexity of $(X,\|\cdot\|)$ it follows that
\begin{equation}\label{JRZ-1}
\lim_{n\to\infty}\|x_{n}-x_{n+1}\|=0.
\end{equation} 
By similar argument we get that $\lim_{n\to\infty}\|y_{n}-y_{n+1}\|=0$.

Let us suppose that there exists $\varepsilon>0$ with the property: for any $j\in\mathbb{N}$ there are $m_j\geq n_j\geq j$ so that
$$\|x_{m_j}-y_{n_j}\|\geq d+\varepsilon.$$
Let us choose $m_j$ to be the smallest integer so that the last inequality is satisfied, i.e.there holds
$$
\|x_{m_j}-y_{n_j}\|\geq d+\varepsilon\ \ \mbox{and}\ \ \|x_{m_j-1}-y_{n_j}\|< d+\varepsilon.
$$
Thus we get
\begin{equation}\label{lemma-3-eq-1}
d+\varepsilon\leq \|x_{m_j}-y_{n_j}\|\leq	\|x_{m_j}-x_{m_j-1}\|+\|x_{m_j-1}-y_{n_j}\|<
\|x_{m_j}-x_{m_j-1}\|+d+\varepsilon.
\end{equation}
Letting $j\to\infty$ in (\ref{lemma-3-eq-1}) by using (\ref{JRZ-1}) we get
$\lim_{j\to\infty}\|x_{m_j}-y_{n_j+1}\|=d+\varepsilon$.
Using the boundedness of $\{x_n\}_{n=0}^\infty$ and $\{y_n\}_{n=0}^\infty$ we get the existence of $M\geq d$, so that the inequality
$M\geq\|x_{0}-y_{m_j-n_j}\|$ 	holds for every $j\in\mathbb{N}$. 
The inequality 
$$
\|x_{m_j}-y_{n_j}\|-d\leq (\alpha+\beta)^{n_j}(\|x_{0}-y_{m_j-n_j}\|-d)\leq
(\alpha+\beta)^{n_j}(M-d)
$$
holds. For any $\varepsilon>0$ we can find $j_0\in\mathbb{N}$ to hold
$(\alpha+\beta)^{j}(M-d)<\varepsilon$ for every $j\geq j_0$.
Therefore for any $m_j\geq n_j\geq j_0$ there holds
$
\|x_{m_j}-x_{n_j}\|<d+\varepsilon$,
which is a contradiction.
\hfill$\Box$

\begin{lemma}\label{lemma-4}
	Let $A_x$, $A_y$ be nonempty convex subsets of a uniformly convex Banach space $(X,\|\cdot\|)$. Let there exist a subset $D\subseteq A_x\times A_y$ and maps $F:D\to A_x$ and $f:D\to A_y$, such that $(F(x,y),f(x,y))\subseteq D$ for every $(x,y)\in D$. 
	Let the ordered pair $(F,f)$ be a cyclic contraction of type two. For an arbitrary chosen $(x,y)\in D$ the sequences $\{x_{n}\}_{n=0}^\infty$ and $\{y_{n}\}_{n=0}^\infty$ are Cauchy.
\end{lemma}

{\it Proof:}
We will prove that $\{x_{n}\}_{n=0}^\infty$ is a Cauchy sequence. The proof for $\{y_{n}\}_{n=0}^\infty$ is similar. By Lemma \ref{lemma-3} we have that
for every $\varepsilon>0$ there is $n_0\in\mathbb{N}$, so that for all $m\geq n\geq n_0$ holds the inequality
$\|x_{m}-y_{n}\|<d+\varepsilon$.

By Lemma \ref{lemma-1} we get
$\lim_{n\to\infty}\|x_{n}-y_{n}\|=d$. According to Lemma \ref{Eldred-1}
it follows that for every $\varepsilon>0$ there is $N_0\in\mathbb{N}$, so that for all $m>n\geq N_0$ holds the inequality
$\|x_{m}-x_{n}\|<\varepsilon$ and consequently $\{x_{n}\}_{n=0}^\infty$ is a Cauchy sequence.
\hfill$\Box$

\begin{lemma}\label{lem:1}
	Let $A_x$, $A_y$be nonempty subsets of a uniformly convex Banach space $(X,\|\cdot\|)$. Let there exist a subset $D\subseteq A_x\times A_y$ and maps $F:D\to A_x$ and $f:D\to A_y$, so that $(F(x,y),f(x,y))\subseteq D$ for every $(x,y)\in D$ and the ordered pair $(F,f)$ be a cyclic contraction of type two.
	Then for an arbitrary chosen $(x,y)\in D$ and for any $1\le l\leq n$ there hold the inequalities
	$\|x_{n}-y_{n}\|\leq (\alpha+\beta)^{l}W_{n-l,n-l}(x,y)+d$.
\end{lemma}
\begin{proof}	
	Using Lemma \ref{lemma-1} we get
	$W_{n,n}(x,y)\leq (\alpha+\beta)W_{n-1,n-1}(x,y)$
	and thus $\|x_{n}-y_{n}\|\leq (\alpha+\beta)^{l}W_{n-l,n-l}(x,y)+d$.
\end{proof}

\begin{lemma}\label{lem:3}
	Let $A_x$, $A_y$ be nonempty closed and convex subsets of a uniformly convex Banach space $(X,\|\cdot\|)$.
	Let there exist a subset $D\subseteq A_x\times A_y$ and maps $F:D\to A_x$ and $f:D\to A_y$, such that $(F(x,y),f(x,y)\subseteq D$ and the ordered pair $(F,f)$ be a cyclic contraction of type two. Then for an arbitrary chosen $(x,y)\in D$
	there holds the inequalities
	$$
	\delta_{\|\cdot\|}\left(\frac{\|x_{n+1}-x_{n}\|}{d+(\alpha+\beta)^lU_{n-l}(x,y)}\right)
	\leq\frac{(\alpha+\beta)^lU_{n-l}(x,y)}{d+(\alpha+\beta)^lU_{n-l}(x,y)}, \ \
	\delta_{\|\cdot\|}\left(\frac{\|y_{n+1}-y_{n}\|}{d+(\alpha+\beta)^lU_{n-l}(y,x)}\right)
	\leq\frac{(\alpha+\beta)^lU_{n-l}(y,x)}{d+(\alpha+\beta)^lU_{n-l}(y,x)},
	$$
	where $U_{n}(x,y)=\max\{W_{n,n+1}(x,y),W_{n,n}(x,y)\}=\max\{\|x_n-y_{n+1}\|-d,\|x_n-y_n\|-d\}$.
\end{lemma}

\begin{proof}
	Using Lemma \ref{lem:1} we obtain
	$$
	\|x_{n}-y_{n+1}\|\leq d+(\alpha+\beta)^lW_{n-l,n+1-l}(x,y)\leq d+(\alpha+\beta)^l\max\{W_{n-l,n+1-l}(x,y),W_{n-l,n-l}(x,y)\}
	$$
	$$
	\|y_{n+1}-x_{n+1}\|\leq d+(\alpha+\beta)^{l+1}W_{n-l,n-l}(x,y)\leq d+(\alpha+\beta)^l\max\{W_{n-l,n+1-l}(x,y),W_{n-l,n-l}(x,y)\}.
	$$
	Then
	$$
	\begin{array}{lll}
	\|x_{n+1}-x_{n}\|&\leq&	\|x_{n+1}-y_n\|+\|y_n-x_{n}\|\leq 2\left(d+(\alpha+\beta)^l\max\{W_{n-l,n+1-l}(x,y),W_{n-l,n-l}(x,y)\}\right)\\
	&=& 2\left(d+(\alpha+\beta)^lU_{n-l,n-l}(x,y)\right).
	\end{array}
	$$
	After a substitution in (\ref{eq:1}) with $x=x_{n}$, $y=y_{n}$, $z=x_{n+1}$,
	$R=d+(\alpha+\beta)^l\max\{W_{n-l,n+1-l}(x,y),W_{n-l,n-l}(x,y)\}$ and $r=\|x_{n+1}-x_{n}\|$
	and from the convexity of $A_x$ we obtain the inequalities
	\begin{equation}\label{eq:9}
	d\leq\left\|\frac{x_{n}+x_{n+1}}{2}-y_{n}\right\|\leq\left(1-\delta_{\|\cdot\|}\left(\frac{\|x_{n}-x_{n+1}\|}{d+(\alpha+\beta)^lU_{n-l,n-l}(x,y)}\right)\right)\left(d+(\alpha+\beta)^lU_{n-l,n-l}(x,y)\right).
	\end{equation}
	Thereafter the inequality
	$\delta_{\|\cdot\|}\left(\frac{\|x_{n+1}-x_{n}\|}{d+(\alpha+\beta)^lU_{n-l,n-l}(x,y)}\right)
	\leq\frac{(\alpha+\beta)^lU_{n-l,n-l}(x,y)}{d+(\alpha+\beta)^lU_{n-l,n-l}(x,y)}$ holds.
\end{proof}

\begin{theorem}\label{th:main}
	Let $A_x$, $A_y$ be nonempty, closed and convex subsets of a uniformly convex Banach space $(X,\|\cdot\|)$. Let there exist a closed and convex subset $D\subseteq A_x\times A_y$ and maps $F:D\to A_x$ and $f:D\to A_y$, such that $(F(x,y),f(x,y)\subseteq D$ for every $(x,y)\in D$. 
	Let the ordered pair $(F,f)$ be a cyclic contraction of type two. Then 
	$(F,f)$ has a unique coupled best proximity point $(\xi,\eta)\in A_x\times A_y$,
	(i.e. $\|\eta-F(\xi,\eta)\|=\|\xi-f(\xi,\eta)\|=d$).
	For any \sout{arbitrary} initial guess $(x,y)\in A_x\times A_y$ there hold $\lim_{n\to\infty}x_{n}=\xi$, $\lim_{n\to\infty}y_{n}=\eta$,
	$\|\xi-\eta\|=d$, $\xi=F(\xi,\eta)$ and $\eta=f(\xi,\eta)$. 
	
	If in addition $(X,\|\cdot\|)$ has a modulus of convexity of power type
	with constants $C>0$ and $q>1$, then
	\begin{enumerate}[\normalfont (i)]
		\item\label{item_ii} a priori error estimates hold
		$$
		\left\|\xi-x_{m}\right\|
		\leq M_0\root{q}\of{\frac{\max\{W_{0,1}(x,y),W_{0,0}(x,y)\}}{Cd}}\cdot\frac{\root {q} \of {(\alpha+\beta)^{m}}}{1-\root {q} \of {\alpha+\beta}};
		\left\|\eta-y_{m}\right\|
		\leq N_0\root{q}\of{\frac{\max\{W_{0,1}(y,x),W_{0,0}(y,x)\}}{Cd}}\cdot\frac{\root {q} \of {(\alpha+\beta)^{m}}}{1-\root {q} \of {\alpha+\beta}};
		$$
		\item\label{item_iii} a posteriori error estimates hold
		$$
		\left\|\xi-x_{n}\right\|
		\leq M_{n-1}\root{q}\of{\frac{\max\{W_{n-1,n}(x,y),W_{n-1,n-1}(x,y)\}}{Cd}}c;
		\left\|\eta-y_{n}\right\|
		\leq N_{n-1}\root{q}\of{\frac{\max\{W_{n-1,n}(y,x),W_{n-1,n-1}(y,x)\}}{Cd}}c,
		$$
		where $M_n=\max\{\|x_{n}-y_{n}\|,\|x_{n}-y_{n+1}\|\}$, $N_n=\max\{\|x_{n}-y_{n}\|,\|y_{n}-x_{n+1}\|\}$ and $c=\frac{\root {q} \of {\alpha+\beta}}{1-\root {q} \of {\alpha+\beta}}$.
	\end{enumerate}
\end{theorem}

\begin{proof}
	For any initial guess $(x,y)\in D$ it follows from Lemma \ref{lemma-4} that $\{x_{n}\}_{n=0}^\infty$ and $\{y_{n}\}_{n=0}^\infty$ are Cauchy sequences. From the assumptions that $(X,\|\cdot\|)$ is a Banach space and $D$ is closed it follows that there are $(\xi,\eta)\in D$, so that
	$\lim_{n\to\infty}x_n=\xi$ and $\lim_{n\to\infty}y_{n}=\eta$.
	
	From the inequalities by using the continuity of the norm function $\|\cdot -\cdot\|$ and Lemma \ref{lemma-1} we have 
	$$
	\begin{array}{lll}
	\|\xi-\eta\|-d&=&\displaystyle\lim_{n\to\infty}\|x_{n}-y_{n}\|-d=\displaystyle\lim_{n\to\infty}\|F(x_{n-1},y_{n-1})-f(x_{n-1},y_{n-1})\|-d\\
	&\leq& \displaystyle\lim_{n\to\infty}\left(\alpha\|x_{n-1}-y_{n-1}\|+\beta\|y_{n-1}-x_{n-1}\|\right)-(\alpha+\beta)d=\displaystyle\lim_{n\to\infty}\left(\alpha+\beta\right)\left(\|x_{n-1}-y_{n-1}\|-d\right)=0.
	\end{array}
	$$
	Thus $\|\xi-\eta\|=d$.
	
	From the inequalities by using the continuity of the norm function $\|\cdot -\cdot\|$ and Lemma \ref{lemma-1} we have
	$$
	\begin{array}{lll}
	\|\xi-f(\xi,\eta)\|-d&=&\displaystyle\lim_{n\to\infty}\|x_{n+1}-f(\xi,\eta)\|-d=\displaystyle\lim_{n\to\infty}\|F(x_{n},y_{n})-f(\xi,\eta)\|-d\\
	&\leq& \displaystyle\lim_{n\to\infty}\left(\alpha\|x_{n}-\eta\|+\beta\|y_{n}-\xi\|\right)-(\alpha+\beta)d=\left(\alpha+\beta\right)\left(\|\xi-\eta\|-d\right)=0.
	\end{array}
	$$
	Thus $\|\xi-f(\xi,\eta)\|=d$. From $\|\xi-\eta\|=d$, according to Lemma \ref{Eldred-2} it follows that $\eta=f(\xi,\eta)$.
	
	From the inequalities by using the continuity of the norm function $\|\cdot -\cdot\|$ and Lemma \ref{lemma-1} have
	$$
	\begin{array}{lll}
	\|\eta-F(\xi,\eta)\|-d&=&\displaystyle\lim_{n\to\infty}\|y_{n+1}-F(\xi,\eta)\|-d=\displaystyle\lim_{n\to\infty}\|f(x_{n},y_{n})-F(\xi,\eta)\|-d\\
	&\leq& \displaystyle\lim_{n\to\infty}\left(\alpha\|x_{n}-\eta\|+\beta\|y_{n}-\xi\|\right)-(\alpha+\beta)d=\left(\alpha+\beta\right)\left(\|\xi-\eta\|-d\right)=0.
	\end{array}
	$$
	Thus $\|\eta-F(\xi,\eta)\|=d$. From $\|\xi-\eta\|=d$, according to Lemma \ref{Eldred-2} it follows that $\xi=F(\xi,\eta)$.
	
	We will prove that the coupled best proximity points are unique.
	
	Let us suppose that there exists $(\xi^{*},\eta^{*})$, such that $\|\eta^{*}-F(\xi^{*},\eta^{*})\|=\|\xi^{*}-f(\xi^{*},\eta^{*})\|=d$
	and $\|\xi-\xi^{*}\|+\|\eta-\eta^{*}\|>0$. From (\ref{eq-cyclic-contraction-1}) we get the inequality
	$$
	\begin{array}{lll}
	\|F(F(\xi^{*},\eta^{*}),f(\xi^{*},\eta^{*}))-f(\xi^{*},\eta^{*})\|
	&\leq&\alpha\|\eta^{*}-F(\xi^{*},\eta^{*})\|+\beta\|\xi^{*}-f(\xi^{*},\eta^{*})\|
	+(1-(\alpha+\beta))d\\
	&=&\alpha d+\beta d+(1-(\alpha+\beta))d=d
	\end{array}
	$$
	From $\|\xi^{*}-f(\xi^{*},\eta^{*})\|=d$, according to Lemma \ref{Eldred-2}
	it follow that $\xi^{*}=F(F(\xi^{*},\eta^{*}),f(\xi^{*},\eta^{*}))$. By analogous arguments we get that $\eta^{*}=f(F(\xi^{*},\eta^{*}),f(\xi^{*},\eta^{*}))$.
	Let us suppose that $\|\xi^{*}-\eta^{*}\|>d$, then
	\begin{equation}\label{8i}
	\begin{array}{lll}
	\|\xi^{*}-\eta^{*}\|&=&\|F(F(\xi^{*},\eta^{*}),f(\xi^{*},\eta^{*}))-f(F(\xi^{*},\eta^{*}),f(\xi^{*},\eta^{*}))\|\\
	&\leq&(\alpha+\beta)\|F(\xi^{*},\eta^{*})-f(\xi^{*},\eta^{*})\|+(1-(\alpha+\beta))d\\
	&\leq&(\alpha+\beta)^2\|\xi^{*}-\eta^{*}\|+(1-(\alpha+\beta))(1+(\alpha+\beta))d\\
	&<&(\alpha+\beta)^2\|\xi^{*}-\eta^{*}\|+(1-(\alpha+\beta))(1+(\alpha+\beta))\|\xi^{*}-\eta^{*}\|<\|\xi^{*}-\eta^{*}\|,
	\end{array}
	\end{equation}
	a contradiction and thus $\|\xi^{*}-\eta^{*}\|=d$. Using that $\|\eta^{*}-F(\xi^{*},\eta^{*})\|=\|\xi^{*}-f(\xi^{*},\eta^{*})\|=d$ and Lemma \ref{Eldred-2}
	we get that $\eta^{*}=f(\xi^{*},\eta^{*})$ and $\xi^{*}=F(\xi^{*},\eta^{*})$.
	Let us suppose that $d<\max\{\|\eta-\xi^{*}\|,\|\xi-\eta^{*}\|\}$. Then
	\begin{equation}\label{8ii}
	\begin{array}{lll}
	\|\xi^{*}-\eta\|&=&\|F(\xi^{*},\eta^{*})-f(\xi,\eta)\|\leq\alpha\|\eta-\xi^{*}\|+\beta\|\xi-\eta^{*}\|+(1-(\alpha+\beta))d\\
	&<&\alpha\|\eta-\xi^{*}\|+\beta\|\xi-\eta^{*}\|+(1-(\alpha+\beta))\displaystyle\frac{\beta\|\xi-\eta^{*}\|+\alpha\|\eta-\xi^{*}\|}{\alpha +\beta}=\displaystyle\frac{\beta\|\xi-\eta^{*}\|+\alpha\|\eta-\xi^{*}\|}{\alpha +\beta}.
	\end{array}
	\end{equation}
	By similar arguments we get
	\begin{equation}\label{8iib}
	\begin{array}{lll}
	\|\xi-\eta^{*}\|&=&\|F(\xi,\eta)-f(\xi^{*},\eta^{*})\|\leq\alpha\|\xi-\eta^{*}\|+\beta\|\eta-\xi^{*}\|+(1-(\alpha+\beta))d\\
	&<&\alpha\|\xi-\eta^{*}\|+\beta\|\eta-\xi^{*}\|+(1-(\alpha+\beta))\displaystyle\frac{\alpha\|\xi-\eta^{*}\|+\beta\|\eta-\xi^{*}\|}{\alpha +\beta}=\displaystyle\frac{\alpha\|\xi-\eta^{*}\|+\beta\|\eta-\xi^{*}\|}{\alpha +\beta}.
	\end{array}
	\end{equation} 
	After summing (\ref{8ii}) and (\ref{8iib}) we get
	\begin{equation}\label{8iii}
	\|\xi^{*}-\eta\|+\|\xi-\eta^{*}\|< \|\xi-\eta^{*}\|+\|\eta-\xi^{*}\|,
	\end{equation}
	a contradiction, i.e. $d=\|\eta-\xi^{*}\|=\|\xi-\eta^{*}\|$. From  Lemma \ref{Eldred-1}, $\|\eta-\xi\|=d$ we obtain that $\xi^{*}=\xi$ and $\eta^{*}=\eta$.
	
	(\ref{item_ii})
	The Uniform convexity of $X$ ensures that $\delta_{\|\cdot\|}$ is strictly increasing and therefore its inverse function
	$\delta^{-1}_{\|\cdot\|}$ exists and is strictly increasing. By Lemma \ref{lem:3} we have
	\begin{equation}\label{eq:11}
	\|x_{n}-x_{n+1}\|
	\leq \left(d+(\alpha+\beta)^lU_{n-l}(x,y)\right)\delta^{-1}_{\|\cdot\|}
	\left(\displaystyle\frac{(\alpha+\beta)^{l}U_{n-l}(x,y)}{d+(\alpha+\beta)^lU_{n-l}(x,y)}\right).
	\end{equation}
	By the inequality $\delta_{\|\cdot\|}(t)\geq Ct^q$ it follows that $\delta^{-1}_{\|\cdot\|}(t)\leq \left(\frac{t}{C}\right)^{1/q}$. From
	(\ref{eq:11}) and the inequalities
	$$
	d\leq d+(\alpha+\beta)^lU_{n-l}(x,y)\leq \max\{P_{n-l,n-l}(x,y),P_{n-l,n-l+1}(x,y)\}
	$$
	we obtain
	\begin{equation}\label{eq:12}
	\begin{array}{lll}
	\|x_{n}-x_{n+1}\|&\leq&
	\left(d+(\alpha+\beta)^lU_{n-l}(x,y)\right)\root{q}\of{\displaystyle\frac{(\alpha+\beta)^{l}U_{n-l}(x,y)}{C\left(d+(\alpha+\beta)^lU_{n-l}(x,y)\right)}}\\
	&\leq&
	\max\{P_{n-l,n-l}(x,y),P_{n-l,n-l+1}(x,y)\}\root{q}\of{\displaystyle\frac{U_{n-l}(x,y)}{Cd}}\root{q}\of{(\alpha+\beta)^l}.
	\end{array}
	\end{equation}
	
	We have proven the existance of a unique pair $(\xi,\eta)\in A_x\times A_y$, so that $\|\xi-F(\xi,\eta)\|=d$, where $\xi$ is a limit of $\{x_{n}\}_{n=1}^\infty$ for any $(x,y)\in A_x\times A_y$.
	
	After a substitution with $l=n$ in (\ref{eq:12}) we get the inequality
	$$
	\begin{array}{lll}
	\displaystyle\sum_{n=1}^\infty\left\|x_{n}-x_{n+1}\right\|&\leq&
	\max\{\|x_0-y_0\|,\|x_0-y_1\|\}\root{q}\of{\displaystyle\frac{U_0(x,y)}{Cd}}\displaystyle\sum_{n=1}^\infty \root {q} \of {(\alpha+\beta)^{n}}\\
	&=&\max\{\|x_0-y_0\|,\|x_0-y_1\|\}\root{q}\of{\displaystyle\frac{U_0(x,y)}{Cd}}\cdot\displaystyle\frac{\root {q} \of {\alpha+\beta}}{1-\root {q} \of {\alpha+\beta}}
	\end{array}
	$$
	and consequently the series $\sum_{n=1}^\infty (x_{n}-x_{n+1})$ is absolutely convergent.
	Consequently for any $m\in\mathbb{N}$ there holds
	$\xi=x_{m}-\sum_{n=m}^\infty \left(x_{n}-x_{n+1}\right)$
	and therefore we get the inequality
	$$
	\left\|\xi-x_{m}\right\|\leq \sum_{n=m}^\infty \left\|x_{n}-x_{n+1}\right\|\leq\max\{\|x_0-y_0\|,\|x_0-y_1\|\}\root{q}\of{\frac{U_0(x,y)}{Cd}}\cdot\frac{\root {q} \of {(\alpha+\beta)^{m}}}{1-\root {q} \of {\alpha+\beta}}.
	$$
	
	The proof for $\|\eta-y_{m}\|$ can be done in a comparative mold.
	
	(\ref{item_iii}) Simply to fit some formulas in the text field we put $M_n=\max\{\|x_n-y_n\|,\|x_n-y_{n+1}\|\}$.
	After substituting in (\ref{eq:12}) with $l=1+i$  we get
	\begin{equation}\label{eq:13}
	\|x_{n+i}-x_{n+i+1}\|\leq
	M_{n-1}\root{q}\of{\frac{U_{n-1}(x,y)}{Cd}}\left(\root{q}\of{\alpha+\beta}\right)^{1+i}.
	\end{equation}
	From (\ref{eq:13}) we get \sout{that there holds} the inequality
	\begin{equation}\label{eq:14}
	\begin{array}{lll}
	\|x_{n}-x_{n+m}\|&\leq&\displaystyle\sum_{i=0}^{m-1}\|x_{n+i}-x_{n+i+1}\|\leq\sum_{i=0}^{m-1}M_{n-1}\root{q}\of{\frac{U_{n-1}(x,y)}{Cd}}\root{q}\of{(\alpha+\beta)^{1+i}}\\
	&=&M_{n-1}\root{q}\of{\displaystyle\frac{U_{n-1}(x,y)}{Cd}}\displaystyle\sum_{i=0}^{m-1} \root {q} \of {(\alpha+\beta)^{1+i}}=M_{n-1}\root{q}\of{\displaystyle\frac{U_{n-1}(x,y)}{Cd}}\cdot
	\displaystyle\frac{1-\root {q} \of {(\alpha+\beta)^{m}}}{1-\root {q} \of {\alpha+\beta}}\root {q} \of {(\alpha+\beta)},
	\end{array}
	\end{equation}
	and after letting $m\to\infty$ in (\ref{eq:14}) we obtain \sout{the inequality}
	$$
	\left\|x_{n}-\xi\right\|\leq
	\max\{\|x_{n-1}-y_{n-1}\|,\|x_{n-1}-y_{n}\|\}\root{q}\of{\frac{U_{n-1}(x,y)}{Cd}}
	\frac{\root {q} \of {\alpha+\beta}}{1-\root {q} \of {\alpha+\beta}}.
	$$
	
	By similar technique we can proof $\|y_{n}-\eta\|$.
\end{proof}

\section{Oligopoly (Duopoly) Markets}
\label{intro}
Markets dominated by a small number of players are getting more common than ever. This process has been fueled by industry consolidation, rise in international expansion and natural desire to benefit from economies in scale, all resulting in a number of mergers and acquisitions that leave few companies dominating a particular market. Cournot in 1838 in in \cite{Cournot} was the first to build a complete model of a market where few players control the price and supply quantity of the goods being traded. The original model is able to correctly estimate equilibrium conditions provided that market participants comply with the following requirements:
\begin{enumerate}[(i)]
	\item there are two players each with sufficient market power to affect the price of the goods being sold;
	\item there is no product differentiation;
	\item decisions on production output are taken simultaneously;
	\item there is no cooperation between market participants and each one reacts in a rational way, seeking to maximize its profit.
\end{enumerate}
If the two companies are not necessarily rational, a different solution can be found as for example in \cite{Ueda, Rubinstein}.

Cournot's approach is known nowadays as a static oligopoly model, which means that each company, participating in the oligopoly market considers the production of  others to remain fixed at least for a given period of time, i.e. player $i$ assumes that in time $t$ the other participants produce the quantities that they have produced in
time $t-1$. In the dynamic case, each company attempts to guess what change of production the other players will make in time $t$ \cite{Cellini-Lambertini}. 

The Cournot's duopoly model can be divided into two kinds - symmetric and asymmetric one \cite{BhanuSinha}. In the latter case we consider an efficient company and a less-efficient one, producing a homogeneous good. Both are asymmetric in terms of their pre--innovation production costs. While both companies may have a fixed marginal cost of production, the efficient one has a lower cost $c_1$, and the less-efficient one has a higher cost $c_2$ - where $c_1<c_2$. Companies still compete in quantities, as in Cournot's duopoly. In the symmetric case both firms are equally efficient. It is important to note under which case a particular market falls, because asymmetric oligopoly may also end up dealing with products that are similar but of different quality. This is a question that falls beyond the scope of our study and we shall assume that there is no significant difference in the quality characteristics of the products produced and sold by each market participant.

Considering the demand, there are two distinct cases that have to be analyzed:
\begin{enumerate}[(i)]
	\item Demand is known and does not change. In this case the Cournot's duopoly solution applies as it is.
	\item Demand may change over time. In case case it is possible to reach a situation in which there is uncertainty about market equilibrium - as in \cite{ZMM}.
\end{enumerate}
We analyze a special case of Cournot's duopoly in which the participating companies face different market demand in each of scenarios.

A generalization of Cournot's model is the Stackelberg duopoly \cite{Anderson-Engers}, where one firm is a leader and the other is a follower. 
This model is applicable when firms choose their output sequentially and not simultaneously. Cournot's model and equilibrium are in fact the direct predecessor of Nash's equilibrium point. Bertrand has introduced another kind of a duopoly model, where firms compete on prices rather than on outputs.

Contemporary markets can be subject to different regulations and barriers. Thus there are constraints applicable that influence the stability of market equilibrium and time required to reach it. We can summarize these constraints as:
\begin{enumerate}[(i)]
	\item The number of market participants;
	
	Starting with the duopoly case in the Cournot’s seminal work \cite{Cournot} it is important to take into account the scalability of suggested solutions, due to the fact that there are various market structures in contemporary economics with high concentration and limited number of players.
	
	\item The interdependence, availability, and access to information;
	
	Companies operating in an environment with high concentration ratios (as measured by share of the largest participants compared to the total market volume) cannot take decisions in a completely isolated way. They need to take into account the effects their decisions have on the other participants, as well as their imminent reaction.
	
	\item The price and non-price competition terms;
	
	In order to account for the actual behavior of companies operating under oligopoly markets, it is necessary to take into consideration the non-price competition. In some cases, product differentiation may not exist (for example in the case of raw materials) but there may be loyalty schemes or aggressive advertisement campaigns that affect the equilibrium in indirect way.
	
	\item Consistency of behavior and time dependence of market conditions;
	
	Solutions that consider time-dependency of company behavior and market changes \cite{BarCoj,ChZhW} are better able to describe contemporary oligopoly markets.
	
	\item Market entry and exit barriers;
	
	Entry and exit barriers can directly influence the number of companies operating on the oligopoly market. They also play important role in shaping the decisions of each participant as barriers can be considered as additional limiting/boundary conditions.
	
	\item Goals and profit maximization behavior;
	
	We assume that profit maximization is the sole purpose of all market participants in the oligopoly markets. It is possible that there are periods of time, or even specific markets in which this is not the case (for example as described in \cite{Klemm}) and the resulting equilibrium is different. In a long term economic agents would need to go back to profit maximization as they may be otherwise subject to acquisition or change in management.
	
	\item Linear and non-linear changes in market conditions and firm behavior.
	
	Taking into consideration these critical factors is very important in order to create formal description of oligopoly markets that is adequate to reality we live in.
\end{enumerate}

Restraints could emerge as milestones at every production stage, at a different scale, size and intensity. Moreover, each of the restraints could influence other, in a different direction and with different strength, and in parallel they could be influenced by the introduced production system and the existing market environment. The market is a vital substance, and the environment in which it functions and “breathes” would challenge the play of both actors, and could perform various scenarios of their action. Hence, these influences could affect the preliminary set up goals and price-polices of each of the observed players.
\subsection{\normalsize Modeling real-world oligopoly markets}

Let us consider two companies that offer identical goods or services. These could range from health--care in a specific region to simple grocery delivery in a neighborhood. While the assumption for having homogeneous goods is quite restrictive, it helps to start with a simple model and then extend it by adding non-price competition and brand loyalty, to name a few extra factors. In support of the approach we have used, it should be noted that oligopoly markets with heterogeneous goods can be analyzed with the same instruments as the ones we employ. The only requirements is to define and estimate parameters of the response function of each market participant. In case of identical goods it is much easier to do this. For complex products that include a variety of factors, such as positioning through non-price attributes, response functions may be harder to define and often composite ones.

To investigate the existence and uniqueness of market equilibrium we employ game theory terminology. This is supported by the notion that, company profits depend on its own output, as well as on the production of the other market participants. Under duopoly markets the result fits naturally into game theory basic cases of strategic interaction. Thus the static Cournot's oligopoly is a fully rational game, based on the \sout{following} assumptions:
\begin{enumerate}[(i)]
	\item each company, in taking its optimal production decision rationally, must know before hand
	all its rival's production and both firms should take their decisions simultaneously;
	\item each firm has a perfect knowledge of the market demand function.
\end{enumerate}

The dynamic model is a game in which case restrictive assumption (i) is replaced by some kind of expectation on the rivals' outputs. While the simplest way is to use naive expectation that production or each market participant will remain at its most recent level, it is also possible to impose more realistic views as in \cite{Teocharis, McManus-Quandt}. As a starting point, lets consider a situation in which there are two players ''A`` and ''B`` producing at moment $n+1$ goods $F(x_n,y_n)$ and $f(x_n,y_n)$, provided that at moment $n$ they have produced $x_n$ and $y_n$ respectively. Such general notation does not yet imply anything regarding market participants. Depending on the functions $F(x_n, y_n)$ and $f(x_n, y_n)$ the model can be static or dynamic, as well as symmetric or asymmetric.

However in order to have market equilibrium, the pair $(x,y)$ should satisfy the equations $x=F(x,y)$ and $y=f(x,y)$. 

Thus we will search for sufficient conditions, depending only on the response functions, that will ensure the existence and uniqueness of equilibrium pair. Compared to the classical approaches in oligopoly markets, this way has several important advantages:
\begin{itemize}
	\item it is possible to account for protective capacity present in contemporary production environments, which allows to have minimal (or even zero) marginal costs within some output ranges;
	\item it is possible to assess whether marker can reach equilibrium, regardless of the initial position;
	\item and finally it is possible to assess time necessary to reach equilibrium and whether this situation can remain stable.
\end{itemize}

An extensive study on the oligopoly markets can be found in \cite{BCKS,MS,Okuguchi,OS}.

\subsection{The basic model}
\label{sec:3}

Let us first start with a duopoly model \cite{Friedman,Smith} - two companies competing for same consumers and striving to meet the demand with overall production of $Z=x+y$. The market price is defined as $P(Z)=P(x+y)$, which is the inverse of the demand function. Market players have cost functions $c_1(x)$ and $c_2(y)$, respectively.
Assuming that both firms are acting rationally, the profit functions are
$\Pi_1(x,y)=xP(x+y)-c_1(x)$ and $\Pi_2(x,y)=yP(x+y)-c_2(y)$ of the first and the second firm, respectively.
The goal of each company is to maximize its profit, i.e. $\max\{\Pi_1(x,y):x,\ \mbox{assuming that}\ y\ \mbox{is fixed}\}$
and $\max\{\Pi_2(x,y):y,\ \mbox{assuming that}\ x\ \mbox{is fixed}\}$.
Provided that functions $P$ and $c_i$, $i=1,2$ are differentiable, we get the equations
\begin{equation}\label{equation:1}
\left|
\begin{array}{l}
\frac{\partial\Pi_1(x,y)}{\partial x}=P(x+y)+xP^\prime(x+y)-c_1^\prime(x)=0\\
\frac{\partial\Pi_2(x,y)}{\partial y}=P(x+y)+yP^\prime(x+y)-c_2^\prime(y)=0.
\end{array}
\right.
\end{equation}

The solution of (\ref{equation:1}) presents the equilibrium pair of production in the duopoly market \cite{Friedman,Smith}.
Often equations (\ref{equation:1}) have solutions in the form of $x=b_1(y)$ and $y=b_2(x)$, which are called response functions \cite{Friedman}.

It may turn out difficult or impossible to solve (\ref{equation:1}) thus it is often advised to search for an approximate solution.
Another drawback, when searching of an approximate solution is that it may be not stable.
Fortunately we can find an implicit formula for the response function in (\ref{equation:1}) i.e.
$x=\frac{c_1^\prime(x)-P(x+y)}{P^\prime(x+y)}=F(x,y) $ and $y=\frac{c_2^\prime(y)-P(x+y)}{P^\prime(x+y)}=f(x,y)$ .

It is still possible that we may end up with response functions, that do not lead to maximization of the profit $\Pi$. As it is often assumed, each participant response depends its own production level and that of the pother payers. E.g. if at a moment $n$ the output quantities are $(x_n,y_n)$, and the first player changes its productions to $x_{n+1}=F(x_n,y_n)$,  then the second one will also change its output to $y_{n+1}=f(x_n,y_n)$. 
The iterated sequence $\{(x_n,y_n)\}_{n=1}^\infty$ are defined in Definition \ref{iterated_sequence}. We have an equilibrium if there are two productions $x$ and $y$, such that $x=F(x,y)$ and $y=f(x,y)$.  The functions $\Pi_i$ are called payoff functions. To ensure that the solutions of (\ref{equation:1}) will present a maximization of the payoff functions a sufficient condition is that $\Pi_i$ be concave functions \cite{BCKS,MS,OS}, by using of response function we alter the maximization problem into a coupled fixed point one thus all assumptions of concavity and differentiability can be skipped. The problem of solving the equations $x=F(x,y)$ and $y=f(x,y)$ is the problem of finding of coupled fixed points for an ordered pair of maps $(F,f)$ \cite{GL}. Yet an important limitation may be that players cannot change output too fast and thus the player may not perform maximize their profits.

Focusing on response functions, allows to put together Cournot and Bertand models. Indeed let the first company have reaction be $F(X,Y)$ and the second one $f(X,Y)$, where $X=(x,p)$ and $Y=(y,q)$. Here $x$ and $y$ denote the output quantity and $(p,q)$ are the prices set by players. In this case companies can compete in terms of both price and quantity.

\section{Application of the Main Results in Duopoly Markets}

In case of two major players taking all or most of the market, we need to consider special cases depending on intersection of production set. The situation in which production sets of both companies have an empty intersection may seem extreme, but it is not impossible. For example if one of the companies is working at a very large scale it may simply be impractical to sustain a low level of output. On the other hand, if the company is just too small to undertake large projects it may also happen that expanding its production beyond certain limit is not feasible. Therefore, it is possible that long term contracts or technical issues prohibit certain type of actions and impose special limitations.
The mathematical justifications of the results is presented in the Appendix section.

\subsection{Players' production sets have a nonempty intersection}

\begin{assumption}\label{assumption1}
	\begin{enumerate}[\normalfont (1)] Let there is a duopoly market, satisfying the following assumptions:
		
		\item The two firms are producing homogeneous goods that are perfect substitutes.
		
		\item The first firm can produce qualities from the set $A_x$ and the second firm can produce qualities from the set $A_y$, where $A_x$ and $A_y$ be closed, nonempty subsets of a complete metric space $(X,\rho)$
		
		\item Let there exist a closed subset $D\subseteq A_x\times A_y$ and maps $F:D\to A_x$ and $f:D\to A_y$, such that $(F(x,y),f(x,y))\subseteq D$ for every $(x,y)\in D$, be the response functions for firm one and two respectively
		
		\item Let there exist $\alpha,\beta,\gamma,\delta>0$, $\max\{\alpha+\gamma,\beta+\delta\}<1$, such that the inequality
		\begin{equation}\label{equation:2}
		\begin{array}{lll}
		\rho(F(x,y),F(u,v))+\rho(f(z,w),f(t,s))&\leq& \alpha \rho(x, u)+\beta \rho(y,v)+\gamma\rho(z, t)+\delta \rho(w,s)
		\end{array}
		\end{equation}
		holds for all $(x,y), (u,v), (z,w), (t,s)\in A_x\times A_y$.
	\end{enumerate}
\end{assumption}

Then
\begin{enumerate}[\normalfont (I)]
	\item\label{item_I-prime} There exists a unique pair $(\xi,\eta)$ in $D$, satisfying $\xi=F(\xi,\eta)$ and $\eta=f(\xi,\eta)$,
	i.e. a market equilibrium pair.
	Moreover the iteration sequences $\{x_{n}\}_{n=0}^\infty$ and $\{y_n\}_{n=0}^\infty$ converge to $\xi$ and $\eta$, respectively.
	\item\label{item_II-prime} a priori error estimates hold
	$\max\left\{\rho(x_{n},\xi),\rho(y_{n},\eta)\right\}\leq \frac{k^{n}}{1-k}(\rho(x_1,x_0)+\rho(y_1,y_0))$;
	\item\label{item_III-prime} a posteriori error estimates hold
	$
	\max\left\{\rho(x_{n},\xi),\rho(y_{n},\eta)\right\}\leq\frac{k}{1-k}(\rho(x_{n-1},x_{n})+\rho(y_{n-1},y_{n}))$;
	\item\label{item_IV-prime} The  rate of convergence for the sequences of successive iterations is 
	$\rho (x_n,\xi)+\rho (y_n,\eta)\leq k\left(\rho (x_{n-1},\xi)+(y_{n-1},\eta)\right)$, where $k=\max\{\alpha+\gamma,\beta+\delta\}$.
\end{enumerate}

The proof is a direct consequence of Theorem \ref{th:3456}.


Let us consider a duopoly market. Let the two firms produce qualities from the set $A_x$ and the second firm can produce qualities from the set $A_y$, where $A_x$ and $A_y$ be nonempty subsets of a complete metric space $(X,\rho)$. Any of the firms can produce a bundle of products $x=(x_1,x_2,\dots x_n)\in X$.
Assumption \ref{assumption1} ensures the existence and uniqueness of the production bundles $(x_1,x_2,\dots x_n), (y_1,y_2,\dots y_n)\in X$ of $n$--goods,
that present the equilibrium in a duopoly economy.

\subsubsection{A linear case, when each player is producing a single product, goods being perfect substitutes}

Let us consider a market with two competing firms, each firm producing just one product, and both goods are perfect substitutes. Let the two firms produce quantities $x\in A_x$ and $y\in A_y$, respectively, where $A_x,A_y\subset [0,+\infty)$ and $(X,\rho)$ be the complete metric space $(\mathbb{R},|\cdot|)$. 
Let us consider the response functions of player one $F(x,y)=a-s-p x-qy$ and player two $f(x,y)=a-r-\mu y-\nu x$,
where 
\begin{enumerate}[\normalfont (1)]
	\item $a,s,r,p,q,\mu,\nu>0$, $s<a$, $r<a$, $\max\{p+\mu,q+\nu\}<1$
	\item $A_x=\left[0,\frac{a-s}{p}\right]\cap \left[0,\frac{a-r}{\mu}\right]$ and $A_y=\left[0,\frac{a-s}{q}\right]\cap \left[0,\frac{a-r}{\nu}\right]$
	\item\label{3a} $D$ can be defined in three ways:
	\begin{enumerate}[\normalfont (\ref{3a}a)]
		\item\label{3aaa} $D=\left[0,\frac{a\mu-aq-s\mu-qr}{\mu p-\nu q}\right]\times \left[0,\frac{ap-a\nu+s\nu-pr}{\mu p-\nu q}\right]$, provided that  $a-s\leq\frac{a\mu-aq-s\mu-qr}{\mu p-\nu q}$ and $a-r\leq\frac{ap-a\nu+s\nu-pr}{\mu p-\nu q}$
		\item\label{3bbb} $D=\left[0,a-s\right]\times \left[0,a-r\right]$, provided that  $\mu r+\nu s-a\mu-a\nu+a-r>0$ and $ps+qr-ap-aq+a-s>0$
		\item\label{3ccc}
		$D=\left\{
		\begin{array}{l}
		0\leq x\leq \displaystyle\frac{a-s}{p}\\
		0\leq y\leq \displaystyle\frac{a-r-\mu x}{\nu}.
		\end{array}
		\right.$
	\end{enumerate}
\end{enumerate}
It is easy to check that $F:D\subset A_x\times A_y\to A_x$, $f:D\subset A_x\times A_y\to A_y$ and $(F(D),f(D))\subseteq D$.

Indeed let us consider case (\ref{3aaa}). From the assumptions that $a,s,r,p,q,\mu,\nu>0$ we get
$$
F(x,y)\leq F(0,0)=a-s\leq \frac{a\mu-aq+s\mu-qr}{\mu p-\nu q},\ \
F(x,y)\geq F\left(\frac{a\mu-aq+s\mu-qr}{\mu p-\nu q},\frac{ap-a\nu+s\nu-pr}{\mu p-\nu q}\right)= 0,
$$
and
$$
f(x,y)\leq f(0,0)=a-r\leq \frac{ap-a\nu+s\nu-pr}{\mu p-\nu q},\ \
f(x,y)\geq f\left(\frac{a\mu-aq+s\mu-qr}{\mu p-\nu q},\frac{ap-a\nu+s\nu-pr}{\mu p-\nu q}\right)= 0.
$$
Therefore $F:D\to A_x$, $f:D\to A_y$ and $(F(D),f(D))\subseteq D$.

It can be proven in a similar fashion that $(F(D),f(D))\subseteq D$ and for the cases (\ref{3bbb}) and (\ref{3ccc}).

From the inequalities $\left|F(x,y)-F(u,v)\right|=\left|p (x-u)+q(y-v)\right|\leq p |x-u|+q |y-v|$ and
$\left|f(z,w)-f(t,s)\right|=\left|\mu (z-t)+\nu (w-s)\right|\leq\mu |z-t|+\nu |w-s|$ it follows that
$$
\left|F(x,y)-F(u,v)\right|+\left|f(z,w)-f(t,s)\right| \leq p |x-u|+q|y-v|+\mu |z-t|+\nu |w-s|
$$
and thus  the ordered pair $(F,f)$ satisfies Assumption \ref{assumption1} with constants $\alpha=p$, $\beta=q$, $\gamma=\mu$, $\delta=\nu$, because $\max\{p+\mu,q+\nu\}<0$.
Consequently there exists an equilibrium pair $(x,y)$ and for any initial start in the economy the iterated sequences $(x_n,y_n)$ converge to the 
market equilibrium $(x,y)$.
The  equilibrium pair is
$x=\frac{a\mu -aq-s\mu+qr+a-s}{\mu p-\nu q+\mu +p+1}$, $y=\frac{ap-a\nu+s\nu -pr -a+r}{\mu p-\nu q+\mu +p+1}$.
\begin{figure}[!h]
	\begin{center}
		\epsfig{file=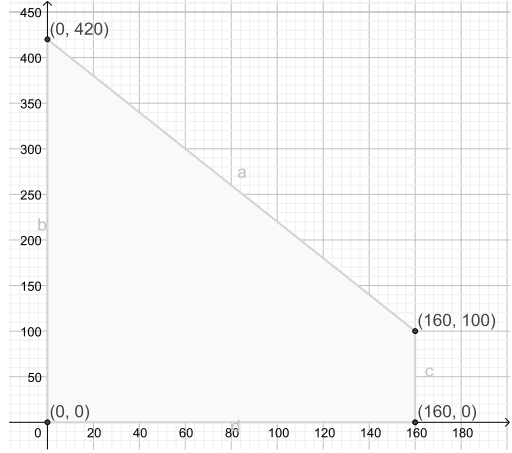,width=0.40\textwidth}
		\caption{The set $D$ in the case \ref{3ccc}}
	\end{center}
\end{figure}
Let us consider a particular case: $a=100$, $s=20$, $r=30$, $p=\frac{1}{2}$, $q=\frac{1}{8}$, $\mu=\frac{1}{3}$, $\nu=\frac{1}{6}$. Values selected for this case are arbitrarily chosen with only general conditions in mind. However in actual situation he values of $p$, $q$, $\mu$ and $\nu$ reflect the actual management and marketing policy of the market participants. 
In this case $F(x,y)=80-\frac{x}{2}-\frac{y}{8}$, $f(x,y)=70-\frac{x}{3}-\frac{y}{6}$, $A_x=[0,210]$, $A_y=[0,320]$. The subset $D$ can be considered either $D=[0,110]\times [0,200]$ (\ref{3aaa}) or
$D=\{0\leq x\leq 160, 0\leq y\leq 420-2x\}$ (\ref{3ccc}).

We get in this case that the equilibrium pair of the production of the two firms is $(49.51, 45.85)$ and the total production will be $x+y=95.36$.

\begin{center}
	\captionof{table}{Values of the iterated sequence $(x_n,y_n)$ if stared with $(40,60)$}
	\label{table1}
	\begin{tabular*}{0.9\textwidth}{@{\extracolsep{\fill} }lrrrrrrr}
		\hline
		$n$&0&1&2&5&10&20&30\\
		\hline
		$x_n$&40&52.5&47.92&49.85&49.49&49.51205&49.51219\\
		\hline
		$y_n$&60&46.6&44.72&46.11&45.83&45.85354&45.85366\\
		\hline
	\end{tabular*}
\end{center}

\begin{center}
	\captionof{table}{Number $n$ of iterations needed by the a priori estimate if stared with $(100,20)$}
	\label{table2}
	\begin{tabular*}{0.9\textwidth}{@{\extracolsep{\fill} }lrrrrr}
		\hline
		$\varepsilon$&0.1&0.01&0.001&0.0001&0.00001\\
		\hline
		$n$&41&53&66&79&91\\
		\hline
	\end{tabular*}
\end{center}

\begin{center}
	\captionof{table}{Number $n$ of iterations needed by the a posteriori estimate if stared with $(100,20)$}
	\label{table3}
	\begin{tabular*}{0.9\textwidth}{@{\extracolsep{\fill} }lrrrrr}
		\hline
		$\varepsilon$&0.1&0.01&0.001&0.0001&0.00001\\
		\hline
		$n$&14&18&23&27&32\\
		\hline
	\end{tabular*}
\end{center}

Let us consider a classical example, where the price function is a linear and so are the cost functions of both players. Assuming the feasible market price is defined by $P(x,y)=120-x-y$, it is expected that additional output $x$ from the first company as well as extra production $y$ of the second one will cause decrease in prices. Therefore under equilibrium conditions $x+y$ will be the total production of the two firms and it will also be reflected in prices. Let the two firms have cost functions equal to $30x$ and $20y$, respectively. The profit of the first one is
$$
\Pi_1(x,y)=xP(x,y)-30x=x(120-x-y)-30x=90x-x^2-xy
$$
and the profit of the second one is
$$
\Pi_2(x,y)=yP(x,y)-20y=y(120-x-y)-20y=100y-y^2-xy.
$$
Following Cournot model after solving (\ref{equation:1}) we get the response functions $F:D\to A_x$ and $f:D\to A_y$ of the two firms
$F(y)=\frac{90-y}{2}$ and $f(x)=\frac{100-x}{2}$,
where $A_y=[0,90]$, $A_x=[0,100]$ and $D=A_x\times A_y$.
Consequently it is a special case of the general example with $a=60$, $s=15$, $r=10$, $p=0$, $q=\frac{1}{2}$, $\mu=\frac{1}{2}$, $\nu=0$.
Thus there exists an equilibrium pair $(x,y)$ and for any initial start in the economy where iterated sequences $(x_n,y_n)$ converge to the 
market equilibrium $(x,y)$. We estimate in this case that the equilibrium pair of the production is $(80/3, 110/3)$ and the total output will be $a=190/3$.

Let us assume that the two firms have started with output $x_0=40$ and $y_0=60$. In the following table we present how, depending on the response functions $F$ and $f$ the output of each company will change.

\begin{center}
	\captionof{table}{Values of the iterated sequence $(x_n,y_n)$ if stared with $(40,60)$}
	\label{table4}
	\begin{tabular*}{0.9\textwidth}{@{\extracolsep{\fill} }lrrrrrr}
		\hline
		$n$&0&1&2&5&10&20\\
		\hline
		$x_n$&40&15&30.0&25.94&26.68&26.67\\
		\hline
		$y_n$&60&30&42.5&36.25&36.69&36.67\\
		\hline
	\end{tabular*}
\end{center}

Let us assume that the two firms have started from productions $x_0=100$ and $y_0=20$. In the next table we present 
how using the response functions $F$ and $f$ the productions of the two firm will change

\begin{center}
	\captionof{table}{Values of the iterated sequence $(x_n,y_n)$ if stared with $(100,20)$}
	\label{table5}
	\begin{tabular*}{0.9\textwidth}{@{\extracolsep{\fill} }lrrrrrr}
		\hline
		$n$&0&1&2&5&10&20\\
		\hline
		$x_n$&100&35&45.0&27.19&26.74&26.67\\
		\hline
		$y_n$&20&0&32.5&34.38&36.65&36.67\\
		\hline
	\end{tabular*}
\end{center}

\begin{center}
	\captionof{table}{Number $n$ of iterations needed by the a priori estimate if stared with $(100,20)$}
	\label{table6}
	\begin{tabular*}{0.9\textwidth}{@{\extracolsep{\fill} }lrrrrr}
		\hline
		$\varepsilon$&0.1&0.01&0.001&0.0001&0.00001\\
		\hline
		$n$&11&15&18&21&25\\
		\hline
	\end{tabular*}
\end{center}

\begin{center}
	\captionof{table}{Number $n$ of iterations needed by the a posteriori estimate if stared with $(100,20)$}
	\label{table7}
	\begin{tabular*}{0.9\textwidth}{@{\extracolsep{\fill} }lrrrrr}
		\hline
		$\varepsilon$&0.1&0.01&0.001&0.0001&0.00001\\
		\hline
		$n$&11&15&18&21&25\\
		\hline
	\end{tabular*}
\end{center}

\subsubsection{A nonlinear case, when each player is producing a single product, while goods sold are perfect substitutes}

Let us consider a market with two competing firms, producing perfect substitute products with quantities $x\in A_x$ and $y\in A_y$, respectively, where $A_x,A_y\subset [0,+\infty)$ and $(X,\rho)$ be the complete metric space $(\mathbb{R},|\cdot|)$. 
Let us assume that each firm produces at least $1$ item, i.e. $x,y\geq 1$
Let us consider the response functions of player one $F(x,y)=\frac{90-x-\frac{y}{8}-\frac{\sqrt{y}}{2}}{2}$ and player 
two $f(x,y)=\frac{100-\frac{x}{4}-y-\sqrt{x}}{3}$,
where 
\begin{enumerate}[\normalfont (1)]
	\item $A_x=[1,44]$ and $A_y=[1,33]$
	\item\label{3b} $D$ can be defined as $D=A_x\times A_y$
\end{enumerate}
It is easy to check that $F:D=A_x\times A_y\to A_x$, $f:D=A_x\times A_y\to A_y$ and $(F(D),f(D))\subseteq D$.

Indeed, we get
$$
F(x,y)\leq F(1,1)=44, F(x,y)\geq F\left(44,33\right)=13.31
$$
and
$$
f(x,y)\leq f(1,1)=32.33<33,\ \ f(x,y)\geq f\left(44,33\right)=5.46
$$
and therefore $F:D\to A_x$, $f:D\to A_y$ and $(F(D),f(D))\subseteq D$.

There exists $\xi$ between the points $y$ and $v$ so that there holds 
$\left|\sqrt{y}-\sqrt{v}\right|=\frac{1}{2\sqrt{\xi}}|y-v|$. 
From the assumption that $y,v\geq 1$ we get that $\left|\sqrt{y}-\sqrt{v}\right|\leq\frac{1}{2}|y-v|$.
Using this last inequality we obtain
$$
\left|F(x,y)-F(u,v)\right|
\leq\frac{1}{2}\left|x-u\right|+\frac{1}{16}\left|y-v\right|+
\frac{1}{8}\left|y-v\right|
=\frac{1}{2}\left|x-u\right|+\frac{3}{16}\left|y-v\right|
$$
and
$$
\left|f(z,w)-f(t,s)\right|
\leq\frac{1}{12}\left|z-t\right|+\frac{1}{3}\left|w-s\right|+
\frac{1}{6}\left|w-s\right|
=\frac{1}{12}\left|z-t\right|+\frac{1}{2}\left|w-s\right|.
$$
Therefore
$$
\left|F(x,y)-F(u,v)\right|+\left|f(z,w)-f(t,s)\right|\leq \displaystyle\frac{1}{2}\left|x-u\right|+\frac{3}{16}\left|y-v\right|+\frac{1}{12}\left|z-t\right|+\frac{1}{2}\left|w-s\right|
$$
and thus  the ordered pair $(F,f)$ satisfies Assumption \ref{assumption1} with constants $\alpha=1/2$, $\beta=3/16$, $\gamma=1/12$ and $\delta=1/2$,
$\max\{1/2+1/12,3/16+1/2\}=\max\{7/12,11/16\}=7/12$.
Consequently there exists an equilibrium pair $(x,y)$ and for any initial start in the economy the iterated sequences $(x_n,y_n)$ converge to the 
market equilibrium $(x,y)$. We get in this case that the equilibrium pair of the production of the two firms is $(28.3, 21.9)$ and the total production will be $a=50.2$.

\begin{center}
	\captionof{table}{Values of the iterated sequence $(x_n,y_n)$ if stared with $(10,50)$}
	\label{table8}
	\begin{tabular*}{0.9\textwidth}{@{\extracolsep{\fill} }lrrrrrrr}
		\hline
		$n$&0&1&2&5&10&20&30\\
		\hline
		$x_n$&10&35.10&25.56&28.61&28.29&28.30747&28.30750\\
		\hline
		$y_n$&50&14.77&23.50&21.99&21.89&21.90064&21.90066\\
		\hline
	\end{tabular*}
\end{center}

\begin{center}
	\captionof{table}{Number $n$ of iterations needed by the a priori estimate if stared with $(10,50)$}
	\label{table9}
	\begin{tabular*}{0.9\textwidth}{@{\extracolsep{\fill} }lrrrrr}
		\hline
		$\varepsilon$&0.1&0.01&0.001&0.0001&0.00001\\
		\hline
		$n$&39&50&62&73&84\\
		\hline
	\end{tabular*}
\end{center}

\begin{center}
	\captionof{table}{Number $n$ of iterations needed by the a posteriori estimate if stared with $(10,50)$}
	\label{table10}
	\begin{tabular*}{0.9\textwidth}{@{\extracolsep{\fill} }lrrrrr}
		\hline
		$\varepsilon$&0.1&0.01&0.001&0.0001&0.00001\\
		\hline
		$n$&12&16&20&24&28\\
		\hline
	\end{tabular*}
\end{center}

Let us consider again a case with two players, producing two products, but let them know the market demand function and let they behave rational, i.e. they a trying to maximize their profits, assuming that the rival player will do the same.

Let there is no limit on the market, but let us assume that the total consumption is $100\%$. That is the market will consume a constant $1$, which is $100\%$, and the production of both firms will be a percentage of the consumption $x$ and $y$ respectively, i.e. $x,y\in [0,1]$. Let the market price be defined by $P(x,y)=1-\frac{x+y}{2}-\frac{x^2+y^2}{24}$, where $x$ is the production of one of the firms, $y$ be the production of the other one, assuming that number $1$ presents $100\%$. Let the two firms have cost functions equal to $C_x(x)=x/2+x^2/16$ and $y/6+y^2/12$, respectively. The profit of the first firm is
$$
\Pi_1(x,y)=xP(x,y)-C_x(x)=\frac{7x}{8}-\frac{9x^2}{16}-\frac{xy}{2}-\frac{x^3}{24}-\frac{xy^2}{24}
$$
and the profit of the second firm is
$$
\Pi_2(x,y)=yP(x,y)-C_y=\frac{5y}{6}-\frac{13y^2}{24}-\frac{xy}{2}-\frac{y^3}{24}-\frac{x^2y}{24}.
$$
Following Cournot model after solving (\ref{equation:1}) we get the response functions $F$ and $f$ of the two players
$$
F(x,y)=\frac{7}{8}-\frac{x}{8}-\frac{y}{2}-\frac{x^2}{8}-\frac{y^2}{24}\ \ \mbox{and}\ \ f(x,y)=\frac{5}{6}-\frac{x}{2}-\frac{y}{12}-\frac{y^2}{8}-\frac{x^2}{24},
$$
which satisfy $F:[0,1]\times [0,1]\to [0,1]$ and $f:[0,1]\times [0,1]\to [0,1]$, i.e. $D=[0,1]\times [0,1]$.
Using the inequality $|x^2-y^2|=2\xi|x-y|\leq 2|x-y|$, for any $x,y\in [0,1]$ and some $\xi$ between $x$ and $y$ we obtain
$$
\left|F(x,y)-F(u,v)\right|
\leq \frac{3}{8}|x-u|+\frac{7}{12}|y-v|
$$
and
$$
\left|f(z,w)-f(s,t)\right|
\leq \frac{7}{12}|z-t|+\frac{7}{24}\left|w-s\right|.
$$
Therefore
$$
\left|F(x,y)-F(u,v)\right|+\left|f(z,w)-f(s,t)\right|\leq\displaystyle \frac{3}{8}\left|x-u\right|+\frac{7}{12}|y-v|+\frac{7}{12}|z-s|+\frac{7}{24}\left|w-t\right|
$$
and thus  the ordered pair $(F,f)$ satisfies Assumption \ref{assumption1} with constants $\alpha=3/8$, $\beta=7/12$, $\gamma=7/12$ and $\delta=7/24$.
There holds $\max\{\alpha+\gamma,\beta+\delta\}<0.958$.
Thus there exists an equilibrium pair $(x,y)$ and for any initial start in the economy the iterated sequences $(x_n,y_n)$ converge to the 
market equilibrium $(x,y)$. We get in this case that the equilibrium pair of the production of the two firms is $(0.537, 0.451)$, i.e. the first firm will have a share of $53.7\%$ and the second one a share of $45.1\%$ of the sold goods. The total production will be $0.989$, i.e. $98.9\%$ of the total demand of the market.

Let us assume that the two firms have started from productions $x_0=40$ and $y_0=60$. In the next table we present 
how using the response functions $F$ and $f$ the productions of the two firms change in time

\begin{center}
	\captionof{table}{Values of the iterated sequence $(x_n,y_n)$ if stared with $(50\%,50\%)$}
	\label{table11}
	\begin{tabular*}{0.9\textwidth}{@{\extracolsep{\fill} }lrrrrrr}
		\hline
		$n$&0&1&2&5&10&20\\
		\hline
		$x_n$&50\%&51.8\%&53.3\%&53.66\%&53.735\%&53.732\%\\
		\hline
		$y_n$&50\%&46.4\%&45.9\%&45.17\%&45.184\%&45.181\%\\
		\hline
	\end{tabular*}
\end{center}

\begin{center}
	\captionof{table}{Values of the iterated sequence $(x_n,y_n)$ if stared with $(10\%,90\%)$}
	\label{table12}
	\begin{tabular*}{0.9\textwidth}{@{\extracolsep{\fill} }lrrrrrr}
		\hline
		$n$&0&1&2&5&10&20\\
		\hline
		$x_n$&10\%&34.6\%&49.2\%&53.16\%&53.749\%&53.732\%\\
		\hline
		$y_n$&90\%&58.4\%&52.4\%&45.19\%&45.201\%&45.181\%\\
		\hline
	\end{tabular*}
\end{center}

\begin{center}
	\captionof{table}{Values of the iterated sequence $(x_n,y_n)$ if stared with $(100\%,0\%)$}
	\label{table13}
	\begin{tabular*}{0.9\textwidth}{@{\extracolsep{\fill} }lrrrrrr}
		\hline
		$n$&0&1&2&5&10&20\\
		\hline
		$x_n$&100\%&66.6\%&61.4\%&53.68\%&53.757\%&53.732\%\\
		\hline
		$y_n$&0\%  &25.0\%&40.6\%&44.55\%&45.201\%&45.181\%\\
		\hline
	\end{tabular*}
\end{center}

\subsubsection{Each player is producing two product types, goods from each type being perfect substitutes}

Let us consider a market with two competing firms, and each firm is producing two product types. For simplicity we assume that goods from each type produced by major players are perfect substitutes. While it is possible that two types have nothing it common, it still means that within each type customers can freely replace product from the first company with one manufactured by the second one. Let us assume that each firm produces at least $1$ item from each product, i.e. $x=(x_1,x_2), y=(y_1,y_2), x_1,x_2,y_1,y_2\geq 1$. Let us denote the production of the two players by $x=(x_1,x_2)$ and $y=(y_1,y_2)$, respectively.

Let the market of the two goods be endowed with the $p$ norm, $p\in [1,\infty)$, i.e.
$$
\rho((x_1,x_2),(y_1,y_2))=\|(x_1,x_2)-(y_1,y_2)\|_p=
\left(|x_1-y_1|^p+|x_2-y_2|^p\right)^{1/p}.
$$

Let us consider the response functions $F(x,y)=(F_1(x,y),F_2(x,y))$ and $f(x,y)=(f_1(x,y),f_2(x,y))$ defined by
$$
F(x,y)=\left\{\begin{array}{l}
\displaystyle\frac{90-\displaystyle\frac{x_1+x_2}{2}-\displaystyle\frac{y_1+y_2}{3}}{3},\\[10pt]
\displaystyle\frac{90-\displaystyle\frac{x_1+x_2}{2}-\displaystyle\frac{y_1+y_2}{3}}{3};
\end{array}
\right.
\ \
f(x,y)=\left\{\begin{array}{l}
\displaystyle\frac{100-\displaystyle\frac{x_1+x_2}{4}-\displaystyle\frac{y_1+y_2}{3}}{4}\\[10pt]
\displaystyle\frac{100-\displaystyle\frac{x_1+x_2}{4}-\displaystyle\frac{y_1+y_2}{3}}{4}.
\end{array}
\right.
$$
where 
\begin{enumerate}[\normalfont (1)]
	\item $A_x=[0,30]\times [0,30]$ and $A_y=[0,25]\times [0,25]$
	\item\label{3ss-0} $D=[0,30]\times[0,30]\times [0,25]\times [0,25]$
\end{enumerate}
It is easy to see that $(F(x,y),f(x,y))\subseteq D$, whenever $(x,y)=((x_1,x_2),(y_1,y_2))\in D$.

Using the inequality $\frac{a+b}{2}\leq \frac{(a^p+b^p)^{1/p}}{2^{1/p}}$, which holds for any $a,b\geq 0$ we get the chain of inequalities
$$
\begin{array}{lll}
\left\|F(x,y)-F(u,v)\right\|_p&=&\left\|\left(F_1(x,y),F_2(x,y)\right)-\left(F_1(u,v),F_2(u,v)\right)\right\|_p\\[16pt]
&=&\displaystyle\left\|\left(\frac{\frac{(u_1+u_2)-(x_1+x_2)}{2}+\frac{v_1+v_2-(y_1+y_2)}{3}}{3},\frac{\frac{(u_1+u_2)-(x_1+x_2)}{2}+\frac{v_1+v_2-(y_1+y_2)}{3}}{3}\right)\right\|_p\\
&\leq&\displaystyle\frac{2}{3}\left(\frac{|u_1-x_1|+|u_2-x_2|}{2}+\frac{|v_1-y_1|+|v_2-y_2|}{3}\right)\\[10pt]
&=&\displaystyle\frac{2}{3}\frac{|u_1-x_1|+|u_2+x_2|}{2}+\frac{4}{9}\frac{|v_1-y_1|+|v_2+y_2|}{2}\\[10pt]
&\leq&\frac{2^{\frac{p-1}{p}}}{3}\left(
|x_1-u_1|^p+|x_2-y_2|^p\right)^{1/p}+\frac{2^{\frac{p-1}{p}+1}}{9}\left(|y_1-v_1|^p+|y_2-v_2|^p\right)^{1/p}\\[10pt]
&=&\frac{2^{\frac{p-1}{p}}}{3}\|(x_1,x_2)-(u_1,u_2)\|_p+\frac{2^{\frac{p-1}{p}+1}}{9}\|(y_1,v_2)-(y_1,v_2)\|_p
=\frac{2^{\frac{p-1}{p}}}{3}\|x-u\|_p+\frac{2^{\frac{p-1}{p}+1}}{9}\|y-v\|_p
\end{array}
$$
and
$$
\begin{array}{lll}
\left\|f(x,y)-f(u,v)\right\|_p&=&\left\|\left(f_1(x,y),f_2(x,y)\right)-\left(f_1(u,v),f_2(u,v)\right)\right\|_p\\[16pt]
&=&\displaystyle\left\|\left(\frac{\frac{(u_1+u_2)-(x_1+x_2)}{4}+\frac{v_1+v_2-(y_1+y_2)}{3}}{3},\frac{\frac{(u_1+u_2)-(x_1+x_2)}{4}+\frac{v_1+v_2-(y_1+y_2)}{3}}{3}\right)\right\|_p\\[10pt]
&\leq&\displaystyle\frac{2}{3}\left(\frac{|u_1-x_1|+|u_2-x_2|}{4}+\frac{|v_1-y_1|+|v_2-y_2|}{3}\right)\\[10pt]
&=&\displaystyle\frac{2}{6}\frac{|u_1-x_1|-|u_2+x_2|}{2}+\frac{4}{9}\frac{|v_1-y_1|+|v_2+y_2|}{2}\\[10pt]
&\leq&\frac{2^{\frac{p-1}{p}}}{6}\left(
|x_1-u_1|^p+|x_2-y_2|^p\right)^{1/p}+\frac{2^{\frac{p-1}{p}+1}}{9}\left(|y_1-v_1|^p+|y_2-v_2|^p\right)^{1/p}\\[10pt]
&=&\frac{2^{\frac{p-1}{p}}}{6}\|(x_1,x_2)-(u_1,u_2)\|_p+\frac{2^{\frac{p-1}{p}+1}}{9}\|(y_1,v_2)-(y_1,v_2)\|_p
=\frac{2^{\frac{p-1}{p}}}{6}\|x-u\|_p+\frac{2^{\frac{p-1}{p}+1}}{9}\|y-v\|_p.
\end{array}
$$
Therefore
$$
\left\|F(x,y)-F(u,v)\right\|+\left\|f(z,w)-f(t,s)\right\|\leq \displaystyle\frac{2^{\frac{p-1}{p}}}{3}\left\|x-u\right\|+\frac{2^{\frac{p-1}{p}+1}}{9}\left\|y-v\right\|+\displaystyle\frac{2^{\frac{p-1}{p}}}{6}\left\|z-t\right\|+\frac{2^{\frac{p-1}{p}+1}}{9}\left\|w-s\right\|.
$$
From the inequalities $\frac{2^{\frac{p-1}{p}}}{3}+\frac{2^{\frac{p-1}{p}}}{6}<\frac{2}{3}+\frac{1}{3}=1$ and
$\frac{2^{\frac{p-1}{p}+1}}{9}+\frac{2^{\frac{p-1}{p}+1}}{9}\leq 2\frac{4}{9}<1$ it follows that
the ordered pair $(F,f)$ satisfies Assumption \ref{assumption1} with constants $\alpha=\frac{2^{\frac{p-1}{p}}}{3}$, $\beta=\frac{2^{\frac{p-1}{p}+1}}{9}$, $\gamma=\frac{2^{\frac{p-1}{p}}}{6}$ and $\delta=\frac{2^{\frac{p-1}{p}+1}}{9}$.
Thus there exists an equilibrium pair $(x,y)$ and for any initial start in the economy the iterated sequences $(x_n,y_n)$ converge to the 
market equilibrium $(x,y)$. We get in this case that the equilibrium pair of the production of the two firms is $x=(19.27, 19.27)$, $y=(19.36,19.36)$ and the total production will be $a=(38.63,38.63)$.

\begin{center}
	\captionof{table}{Values of the iterated sequence $(x_n,y_n)$ if started with $x=(19.27, 19.27)$, $y=(19.36,19.36)$}
	\label{table14}
	\begin{tabular*}{0.9\textwidth}{@{\extracolsep{\fill} }lrrr}
		\hline
		$n$&0&1&2\\
		\hline
		$x_n$&(10,10)&(15.56,15.56)&(21.39,21.39)\\
		\hline
		$y_n$&(50,50)&(15.42,15.42)&(20.49,20.49)\\
		\hline
	\end{tabular*}
\end{center}
\begin{center}
	\captionof{table}{Values of the iterated sequence $(x_n,y_n)$ if started with $x=(19.27, 19.27)$, $y=(19.36,19.36)$}
	\label{table14-1}
	\begin{tabular*}{0.9\textwidth}{@{\extracolsep{\fill} }lrrr}
		\hline
		$n$&5&10&20\\
		\hline
		$x_n$&(19.09,19.09)&(19.28,19.28)&(19.27,19.27)\\
		\hline
		$y_n$&(19.28,19.28)&(19.36,19.36)&(19.36,19.36)\\
		\hline
	\end{tabular*}
\end{center}

\begin{center}
	\captionof{table}{Number $n$ of iterations needed by the a priori estimate if stared with $x=(19.27, 19.27)$, $y=(19.36,19.36)$ and $p=2$}
	\label{table15}
	\begin{tabular*}{0.9\textwidth}{@{\extracolsep{\fill} }lrrrrr}
		\hline
		$\varepsilon$&0.1&0.01&0.001&0.0001&0.00001\\
		\hline
		$n$&16&21&26&31&36\\
		\hline
	\end{tabular*}
\end{center}

\begin{center}
	\captionof{table}{Number $n$ of iterations needed by the a posteriori estimate if stared with $x=(19.27, 19.27)$, $y=(19.36,19.36)$ and $p=2$}
	\label{table16}
	\begin{tabular*}{0.9\textwidth}{@{\extracolsep{\fill} }lrrrrr}
		\hline
		$\varepsilon$&0.1&0.01&0.001&0.0001&0.00001\\
		\hline
		$n$&9&12&15&18&20\\
		\hline
	\end{tabular*}
\end{center}

\subsection{The players are producing a single product and compete on both quantities and prices}

There is a large number of goods where companies can compete on both quality and prices. In this case the equilibrium would depend on balanced decision on what market share to target at a reasonable price. Lets assume that there are only two major players that produce homogeneous products.
The first company can produce qualities from the set $A_x\subseteq [0,\infty)$ at a price $p\in P_x\subseteq [0,\infty)$ and the second one can produce qualities from the set $A_y\subseteq [0,\infty)$ at a price $p\in P_x\subseteq [0,\infty)$, where $A_x$, $A_y$, $P_x$, $P_y$ be nonempty subsets. Let $A_x\times P_x$, $A_y\times P_y$ be subsets of a complete metric space $(\mathbb{R}^2,\rho)$.

\begin{assumption}\label{assumption2}
	\begin{enumerate}[\normalfont (1)] Let there is a duopoly market, satisfying the following assumptions:
		
		\item The two firms are producing homogeneous, perfect substitute products.
		
		\item The first firm can produce qualities from the set $A_x$ at a price $p\in P_x$ and the second firm can produce qualities from the set $A_y$ at a price $p\in P_x$, where $A_x\times P_x$, $A_y\times P_y$ be nonempty, closed subsets of a complete metric space $(\mathbb{R}^2,\rho)$.
		
		\item Let there exists a closed subset $D\subseteq A_x\times P_x\times A_y\times P_y\to A_x$, such that $F:D\to A_x\times P_x$, $f:D\to A_y\times P_y$ and $(F(x,p,y,q),f(x,p,y,q))\subseteq D$ for every $(x,p,y,q)\in D$ be the response functions for firm one and two respectively.
		
		\item Let there exist $\alpha,\beta,\gamma,\delta>0$, $\max\{\alpha+\gamma,\beta+\delta\}<1$, such that the inequality
		\begin{equation}\label{equation:6}
		\begin{array}{lll}
		S_1&=&\rho(F(x,p_1,y,q_1),F(u,p_2,v,q_2))+\rho(f(z,p_3,w,q_3),f(t,p_4,s,q_4))\\
		&\leq& \alpha \rho((x,p_1), (u,p_2))+\beta \rho((y,q_1),(v,q_2))+\gamma\rho((z,p_3), (t,p_4))+\delta \rho((w,q_3),(s,q_4))
		\end{array}
		\end{equation}
		holds for all $(x,p_1,y,q_1), (u,p_2,v,q_2), (z,p_3,w,q_3), (t,p_4,s,q_4)\in D$.
	\end{enumerate}
\end{assumption}

Then
\begin{enumerate}[\normalfont (I)]
	\item\label{item_I-2} There exists a unique pair $(\xi,p,\eta,q)$ in $A_x\times P_x\times A_y\times P_y$, which is a common coupled fixed point for the maps $F$ and $f$,
	i.e. a market equilibrium pair.
	Moreover the iteration sequences $\{x_{n}\}_{n=0}^\infty$, $\{p_{n}\}_{n=0}^\infty$, $\{y_n\}_{n=0}^\infty$ and $\{q_{n}\}_{n=0}^\infty$
	converge to $\xi$, $p$, $\eta$, and $q$, respectively.
	\item\label{item_II-prime-2} a priori error estimates hold
	\begin{equation}\label{equation:8}
	\max\left\{\rho((x_{n},p_n),(\xi,p)),\rho((y_{n},q_n),(\eta,q))\right\}
	\leq \frac{k^{n}}{1-k}(\rho((x_1,p_1),(x_0,p_0))+\rho((y_1,q_1),(y_0,q_0)));
	\end{equation}
	\item\label{item_III-prime-2} a posteriori error estimates hold
	\begin{equation}\label{equation:9}
	\max\left\{\rho((x_{n},p_n),(\xi,p)),\rho((y_{n},q_n),(\eta,q))\right\}
	\leq\frac{k}{1-k}(\rho((x_{n-1},p_{n-1}),(x_{n},p_n))+\rho((y_{n-1},p_{n-1}),(y_{n},q_n)));
	\end{equation}
	\item\label{item_IV-prime-2} The  rate of convergence for the sequences of successive iterations is given by
	\begin{equation}\label{equation:10}
	\rho ((x_n,p_n),(\xi,p))+\rho ((y_n,q_n),(\eta,q))
	\leq k\left(\rho ((x_{n-1},p_{n-1},(\xi,p))+\rho((y_{n-1},p_{n-1}),(\eta,q))\right),
	\end{equation}
	where $k=\max\{\alpha+\gamma,\beta+\delta\}$.
\end{enumerate}

The proof is a direct consequence of Theorem \ref{th:3456}.

{\bf Remark:} If we put $X=(x,p_1)$, $Y=(y,q_1)$, $U=(u,p_2)$, $V=(v,q2)$, $Z=(z,p_3)$, $W=(w,q_3)$, $T=(t,p_4)$, $S=(s,q_4)$ then
(\ref{equation:6}) can be written in the form
\begin{equation}\label{equation:7}
\rho(F(X,Y),F(U,V))+\rho(f(Z,W),f(T,S))\leq\alpha \rho(X,U)+\beta \rho(Y,V)+\gamma\rho(Z,T)+\delta \rho(W,S)
\end{equation}

\subsubsection{Example of a duopoly model, where players compete on quantities and prices simultaneously}

Let us consider a market with two competing firms, producing the same product, and selling it at a price $p$ and $q$ respectively, i.e. $X=(x,p), Y=(y,q)$.
Let us consider the response functions $F(X,Y)=(F_1(X,Y),F_2(X,Y))$ and $f(X,Y)=(f_1(X,Y),f_2(X,Y))$ defined by
$$
F(X,Y)=\left\{\begin{array}{l}
\displaystyle\frac{90-\displaystyle\frac{x}{2}-\displaystyle\frac{y}{3}}{3},\\[10pt]
\displaystyle\frac{4-\displaystyle\frac{p}{2}-\displaystyle\frac{q}{3}}{3};
\end{array}
\right.
\ \
f(X,Y)=\left\{\begin{array}{l}
\displaystyle\frac{100-\displaystyle\frac{x}{4}-\displaystyle\frac{y}{3}}{4}\\[10pt]
\displaystyle\frac{5-\displaystyle\frac{p}{4}-\displaystyle\frac{q}{3}}{4}.
\end{array}
\right.
$$
Let $X=(x,p)$ and $Y=(y,q)$ be subsets of $(\mathbb{R}^2,\|\cdot\|_2)$ (the two dimensional Euclidean space).
Let
\begin{enumerate}[\normalfont (1)]
	\item $A_x=[0,100]\times [0,5]$ and $A_y=[0,100]\times [0,4]$
	\item\label{3ss} $D=[0,100]\times [0,5]\times [0,100]\times [0,4]$
\end{enumerate}
It is easy to see that $F:D\to [0,100]\times [0,5]$, $f:D\to [0,100]\times [0,4]$ and $(F(x,p,y,q),f(x,p,y,q))\subseteq D$ for every $(x,p,y,q)\in D$.

Using the inequality $\frac{a+b}{2}\leq \frac{(a^2+b^2)^{1/2}}{2^{1/2}}$, which holds for any $a,b\geq 0$ we get \sout{the chain of inequalities we} obtain
$$
\begin{array}{lll}
S_{2}&=&\left\|F(x,p_1,y,q_1)-F(u,p_2,v,q_2)\right\|_2=\left\|\left(\displaystyle\frac{90-\displaystyle\frac{x}{2}-\displaystyle\frac{y}{3}}{3},\displaystyle\frac{4-\displaystyle\frac{p_1}{2}-\displaystyle\frac{q_1}{3}}{3}\right)-\left(\displaystyle\frac{90-\displaystyle\frac{u}{2}-\displaystyle\frac{v}{3}}{3},\displaystyle\frac{4-\displaystyle\frac{p_2}{2}-\displaystyle\frac{q_2}{3}}{3}\right)\right\|_2\\[14pt]
&=&\left\|\left(\displaystyle\frac{\displaystyle\frac{u-x}{2}+\displaystyle\frac{v-y}{3}}{3},\displaystyle\frac{\displaystyle\frac{p_1-p_2}{2}+\displaystyle\frac{q_1-q_2}{3}}{3}\right)\right\|_2
\leq\displaystyle\frac{1}{3}\sqrt{\displaystyle\left(\frac{|u-x|+|v-y|}{2}\right)^2+\displaystyle\left(\frac{|p_1-p_2|+|q_1-q_2|}{2}\right)^2}\\[14pt]
&\leq&\displaystyle\frac{1}{3}\left(\displaystyle\frac{|u-x|+|v-y|}{2}+\displaystyle\frac{|p_1-p_2|+|q_1-q_2|}{2}\right)
=
\displaystyle\frac{1}{3}\left(\displaystyle\frac{|u-x|+|p_1-p_2|}{2}+\displaystyle\frac{|v-y|+|q_1-q_2|}{2}\right)\\[14pt]
&\leq&\displaystyle\frac{1}{3\sqrt{2}}\sqrt{|u-x|^2+|p_1-p_2|^2}+\displaystyle\frac{1}{3\sqrt{2}}\sqrt{|v-y|+|q_1-q_2|}
=\displaystyle\frac{1}{3\sqrt{2}}\rho (U,X)+\displaystyle\frac{1}{3\sqrt{2}}\rho(V,Y)
\end{array}
$$
and
$$
\begin{array}{lll}
S_{3}&=&\left\|f(x,p_1,y,q_1)-f(u,p_2,v,q_2)\right\|_2=\left\|\left(\displaystyle\frac{100-\displaystyle\frac{x}{4}-\displaystyle\frac{y}{3}}{4},\displaystyle\frac{5-\displaystyle\frac{p_1}{4}-\displaystyle\frac{q_1}{3}}{4}\right)-\left(\displaystyle\frac{100-\displaystyle\frac{u}{4}-\displaystyle\frac{v}{3}}{4},\displaystyle\frac{5-\displaystyle\frac{p_2}{4}-\displaystyle\frac{q_2}{3}}{4}\right)\right\|_2\\[14pt]
&=&\left\|\left(\displaystyle\frac{\displaystyle\frac{u-x}{4}+\displaystyle\frac{v-y}{3}}{4},\displaystyle\frac{\displaystyle\frac{p_1-p_2}{4}+\displaystyle\frac{q_1-q_2}{3}}{4}\right)\right\|_2\leq\displaystyle\frac{1}{4}\sqrt{\displaystyle\left(\frac{|u-x|+|v-y|}{3}\right)^2+\displaystyle\left(\frac{|p_1-p_2|+|q_1-q_2|}{3}\right)^2}\\[14pt]
&\leq&\displaystyle\frac{1}{4}\left(\displaystyle\frac{|u-x|+|v-y|}{3}+\displaystyle\frac{|p_1-p_2|+|q_1-q_2|}{3}\right)
<
\displaystyle\frac{1}{4}\left(\displaystyle\frac{|u-x|+|p_1-p_2|}{2}+\displaystyle\frac{|v-y|+|q_1-q_2|}{2}\right)\\[14pt]
&\leq&\displaystyle\frac{1}{4\sqrt{2}}\sqrt{|u-x|^2+|p_1-p_2|^2}+\displaystyle\frac{1}{4\sqrt{2}}\sqrt{|v-y|+|q_1-q_2|}
=\displaystyle\frac{1}{4\sqrt{2}}\rho (U,X)+\displaystyle\frac{1}{4\sqrt{2}}\rho(V,Y).
\end{array}
$$
Therefore
$$
\left\|F(X,Y)-F(U,V)\right\|+\left\|f(Z,W)-f(T,S)\right\|\leq \displaystyle\frac{1}{3\sqrt{2}}\left\|X-U\right\|+\frac{1}{3\sqrt{2}}\left\|Y-V\right\|+\displaystyle\frac{1}{4\sqrt{2}}\left\|Z-T\right\|+\frac{1}{4\sqrt{2}}\left\|W-S\right\|.
$$
From the inequalities $\frac{1}{3\sqrt{2}}+\frac{1}{4\sqrt{2}}<1$ it follows that
the ordered pair $(F,f)$ satisfies Assumption \ref{assumption2} with constants $\alpha=\frac{1}{3\sqrt{2}}$, $\beta=\frac{1}{3\sqrt{2}}$, $\gamma=\frac{1}{4\sqrt{2}}$ and $\delta=\frac{1}{4\sqrt{2}}$.
Thus there exists an equilibrium pair $(x,y)$ and for any initial start in the economy the iterated sequences $(x_n,y_n)$ converge to the 
market equilibrium $(x,y)$. We get in this case that the equilibrium pair of the production of the two firms is $x=(23.64, 1.03)$, $y=(21.71,1.09)$.

\subsection{Players' production sets have an empty intersection}

\begin{assumption}\label{assumption3}
	\begin{enumerate}[\normalfont (1)] Let there is a duopoly market, satisfying the following assumptions:
		
		\item The two firms are producing homogeneous perfect substitute products.
		
		\item The first firm can produce qualities from the set $A_x$ and the second firm can produce qualities from the set $A_y$, where $A_x$ and $A_y$ be nonempty, closed and convex subsets of a uniformly convex Banach space $(X,\|\cdot\|)$
		
		\item Let there exist a closed and convex subset $D\subseteq A_x\times A_y$ and maps $F:D\to A_x$ and $f:D\to A_y$, such that $(F(x,y),f(x,y))\subseteq D$ for every $(x,y)\in D$, be the response functions for firm one and two respectively
		
		\item Let there exist $\alpha,\beta>0$, $\alpha+\beta<1$, such that
		\begin{equation}\label{equation:11}
		\|F(x,y)-f(u,v)\|\leq \alpha \|x-v\|+\beta \|y-u\|+(1-(\alpha+\beta))d
		\end{equation}
		for all $(x,y),(u,v)\in A_x\times A_y$, where $d={\rm dist}(A_x,A_y)=\inf\{\|x-y\|:x\in A_x,y\in A_y\}$.
	\end{enumerate}
\end{assumption}

Then there exists a unique pair $(\xi,\eta)$ in $A_x\times A_y$, which is a coupled best point for the pair of maps $(F,f)$,
i.e. a market equilibrium pair.
Moreover the iteration sequences $\{x_{n}\}_{n=0}^\infty$ and $\{y_n\}_{n=0}^\infty$\sout{, defined in Definition \ref{iterated_sequence}} converge to $\xi$ and $\eta$, respectively.

If in addition $(X,\|\cdot\|)$ has a modulus of convexity of power type
with constants $(C,q)\in (0,\infty)\times (1,\infty)$, then
\begin{enumerate}[\normalfont (i)]
	\item\label{item_ii*-0} a priori error estimates hold
	$$
	\left\|\xi-x_{m}\right\|
	\leq M_0\root{q}\of{\frac{\max\{W_{0,1}(x,y),W_{0,0}(x,y)\}}{Cd}}\cdot\frac{\root {q} \of {(\alpha+\beta)^{m}}}{1-\root {q} \of {\alpha+\beta}};
	\left\|\eta-y_{m}\right\|
	\leq N_0\root{q}\of{\frac{\max\{W_{0,1}(y,x),W_{0,0}(y,x)\}}{Cd}}\cdot\frac{\root {q} \of {(\alpha+\beta)^{m}}}{1-\root {q} \of {\alpha+\beta}};
	$$
	\item\label{item_iii-0} a posteriori error estimates hold
	$$
	\left\|\xi-x_{n}\right\|
	\leq M_{n-1}\root{q}\of{\frac{\max\{W_{n-1,n}(x,y),W_{n-1,n-1}(x,y)\}}{Cd}}
	c.
	\left\|\eta-y_{2n}\right\|
	\leq N_{n-1}\root{q}\of{\frac{\max\{W_{n-1,n}(y,x),W_{n-1,n-1}(y,x)\}}{Cd}}
	c,
	$$
	where $W_{n,m}(x,y)=\|x_n-x_m\|-d$, $M_n\max\{\|x_n-y_n\|,\|x_n-y_{n+1}\|\}$, $N_n\max\{\|x_n-y_n\|,\|y_n-x_{n+1}\|\}$ and $c=\frac{\root {q} \of {\alpha+\beta}}{1-\root {q} \of {\alpha+\beta}}$.
\end{enumerate}

The proof is a direct consequence of Theorem \ref{th:main}.

\subsubsection{Players' production sets have an empty intersection, each player is producing two goods}

Let us consider a market with two competing firms, each firm produces two product and any one of the items is completely replaceable with the similar product of the other firm. Let us assume that the first firm can produce
much less quantities than the second one, i.e. if $x_1,x_2$ be the quantities produced by the first firm and $y_1,y_2$ be the quantities produced by the second one and, then $x_1,x_2\in [0,1]$ and $y_1,y_2\in [2,3]$. Let $A_x=[0,1]\times [0,1]$ $A_y=[2,3]\times [2,3]$ be considered as subsets of $(\mathbb{R}^2,\|\cdot\|_2)$,
which is a uniformly convex Banach space with modulus of convexity $\delta_{\|\cdot\|_2}(\varepsilon)\geq \frac{\varepsilon^2}{8}$ of power type \cite{Z}.
Let us consider the response functions $F(x_1,x_2,y_1,y_2)$ and $f(x_1,x_2,y_1,y_2)$ defined by
$$
F(x,y)=\left\{
\begin{array}{l}
\frac{3x_1}{8}+\frac{x_2}{8}-\frac{3y_1}{16}-\frac{y_2}{16}+1\\
\frac{x_1}{8}+\frac{3x_2}{8}-\frac{y_1}{16}-\frac{3y_2}{16}+1
\end{array}
\right. ,
\ \
f(x,y)=\left\{
\begin{array}{l}
-\frac{3x_1}{16}-\frac{x_2}{16}+\frac{3y_1}{4}+\frac{y_2}{4}+\frac{5}{4}\\
-\frac{x_1}{16}-\frac{3x_2}{16}+\frac{y_1}{4}+\frac{3y_2}{4}+\frac{5}{4}
\end{array}
\right. .
$$
It is easy to see that $F:[0,1]\times [0,1]\times [2,3]\times [2,3]\to [0,1]\times[0,1]$ and $f:[0,1]\times [0,1]\times [2,3]\times [2,3]\to [2,3]\times[2,3]$

Indeed the inequalities $0\leq\frac{3x_1}{8}+\frac{x_2}{8}-\frac{3y_1}{16}-\frac{y_2}{16}+1\leq 1$ are equivalent to
$$
\left|
\begin{array}{rcl}
\frac{3y_1}{16}+\frac{y_2}{16}&\leq& \frac{3x_1}{8}+\frac{x_2}{8}+1\\[10pt]
\frac{3x_1}{8}+\frac{x_2}{8}&\leq& \frac{3y_1}{16}+\frac{y_2}{16}
\end{array}
\right.
$$
for $(x_1,x_2,y_1,y_2)\in [0,1]\times [0,1]\times [2,3]\times [2,3]$. 

The inequalities $2\leq-\frac{3u_1}{16}-\frac{u_2}{16}+\frac{3v_1}{4}+\frac{v_2}{4}+\frac{5}{4}\leq 3$ are equivalent to
$$
\left|
\begin{array}{rcl}
\frac{3u_1}{16}+\frac{u_2}{16}&\leq&\frac{3v_1}{4}+\frac{v_2}{4}\\[10pt]
\frac{3v_1}{4}+\frac{v_2}{4}&\leq& \frac{7}{4}+\frac{3u_1}{16}+\frac{u_2}{16}
\end{array}
\right.
$$
for $(u_1,u_2,v_1,v_2)\in [0,1]\times [0,1]\times [2,3]\times [2,3]$.  

Using the inequality $\left(\frac{3a}{4}+\frac{b}{4}\right)^2\leq\frac{3}{4}a^2+\frac{1}{4}b^2$, 
i.e. $\left|\frac{3a}{4}+\frac{b}{4}\right|\leq\frac{\sqrt{3a^2+b^2}}{2}$ we obtain
$$
\begin{array}{lll}
S_{4}&=&\|F(x_1,x_2,y_1,y_2)-f(u_1,u_2,v_1,v_2)\|_2\\[10pt]
&=&\left\|\left(\frac{3x_1}{8}+\frac{x_2}{8}-\frac{3y_1}{16}-\frac{y_2}{16}+1,\frac{x_1}{8}+\frac{3x_2}{8}-\frac{y_1}{16}-\frac{3y_2}{16}+1\right)-\left(-\frac{3u_1}{16}-\frac{u_2}{16}+\frac{3v_1}{4}+\frac{v_2}{4}+\frac{5}{4},-\frac{u_1}{16}-\frac{3u_2}{16}+\frac{v_1}{4}+\frac{3v_2}{4}+\frac{5}{4}\right)\right\|_2\\[10pt]
&=&\left\|\left(\frac{1}{4}+\frac{3(x_1-v_1)}{8}+\frac{x_2-v_2}{8},\frac{1}{4}+\frac{3(u_1-y_1)}{16}+\frac{u_2-y_2}{16}\right)\right\|_2
\leq\left\|\left(\frac{1}{4},\frac{1}{4}\right)\right\|_2+\left\|\left(\frac{3(x_1-v_1)}{8}+\frac{x_2-v_2}{8},\frac{3(u_1-y_1)}{16}+\frac{u_2-y_2}{16}\right)\right\|_2\\[10pt]
&\leq&\frac{\sqrt{2}}{4}+\left\|\left(\frac{3(x_1-v_1)}{8}+\frac{x_2-v_2}{8},0\right)\right\|_2+\left\|\left(0,\frac{3(u_1-y_1)}{16}+\frac{u_2-y_2}{16}\right)\right\|_2\\[10pt]
&=&\frac{\sqrt{2}}{4}+\frac{1}{2}\left\|\left(\frac{3(x_1-v_1)}{4}+\frac{x_2-v_2}{4},0\right)\right\|_2+\frac{1}{4}\left\|\left(0,\frac{3(u_1-y_1)}{4}+\frac{u_2-y_2}{4}\right)\right\|_2\\[10pt]
&=&\frac{\sqrt{2}}{4}+\frac{1}{2}\left|\frac{3(x_1-v_1)}{4}+\frac{x_2-v_2}{4}\right|+\frac{1}{4}\left|\frac{3(u_1-y_1)}{4}+\frac{u_2-y_2}{4}\right|_2\leq\frac{\sqrt{2}}{4}+\frac{\sqrt{3|x_1-v_1|^2+|x_2-v_2|^2}}{4}+\frac{\sqrt{3|u_1-y_1|^2+|u_2-y_2|^2}}{8}\\[10pt]
&\leq&\frac{\sqrt{2}}{4}+\frac{\sqrt{3}}{4}\sqrt{|x_1-v_1|^2+|x_2-v_2|^2}+\frac{\sqrt{3}}{8}\sqrt{|u_1-y_1|^2+|u_2-y_2|^2}
=\frac{\sqrt{3}}{4}\|x-v\|_2+\frac{\sqrt{3}}{8}\|y-u\|_2+\frac{\sqrt{2}}{4}\\[10pt]
&\leq&\frac{\sqrt{3}}{4}\|x-v\|_2+\frac{\sqrt{3}}{8}\|y-u\|_2+\left(1-\frac{\sqrt{3}}{4}-\frac{\sqrt{3}}{8}\right)\sqrt{2}
=\frac{\sqrt{3}}{4}\|x-v\|_2+\frac{\sqrt{3}}{8}\|y-u\|_2+\left(1-\frac{\sqrt{3}}{4}-\frac{\sqrt{3}}{8}\right)d,
\end{array}
$$
where $d={\rm dist}([0,1]\times [0,1],[2,3]\times [2,3])=\sqrt{2}$.
Therefore
the ordered pair $(F,f)$ satisfies Assumption \ref{assumption3} with constants $\alpha=\frac{\sqrt{3}}{4}$, $\beta=\frac{\sqrt{3}}{8}$.
Thus there exists an equilibrium pair $(x,y)=((x_1,x_2),(y_1,y_2))$ and for any initial start in the economy, the iterated sequence 
$(x^n,y^n)=((x_1^n,x_2^n),(y_1^n,y_2^n))$ converges to the 
market equilibrium $(x,y)$. We get in this case that the equilibrium pair of the production of the two firms is $x=(1,1)$, $y=(2,2)$ and the total production will be $a=(3,3)$.

\begin{center}
	\captionof{table}{Values of the iterated sequence $(x^n,y^n)$ if stared with $((0.01,0.2),(2.9,2.1))$}
	\label{table17}
	\begin{tabular*}{0.9\textwidth}{@{\extracolsep{\fill} }lrrr}
		\hline
		$n$&0&1&2\\
		\hline
		$(x^n_1,x^n_2)$&(0.01,0.9)&(0.44,0.76)&(0.66,0.75)\\
		\hline
		$(y^n_1,y^n_2)$&(2.90,2.1)&(2.44,2.33)&(2.31,2.27)\\
		\hline
	\end{tabular*}
\end{center}

\begin{center}
	\captionof{table}{Values of the iterated sequence $(x^n,y^n)$ if stared with $((0.01,0.2),(2.9,2.1))$}
	\label{table17-1}
	\begin{tabular*}{0.9\textwidth}{@{\extracolsep{\fill} }lrrr}
		\hline
		$n$5&10&20\\
		\hline
		$(x^n_1,x^n_2)$&(0.87,0.88)&(0.97,0.97)&(1,1)\\
		\hline
		$(y^n_1,y^n_2)$&(2.12,2.12)&(2.03,2.03)&(2,2)\\
		\hline
	\end{tabular*}
\end{center}

\subsection{Equilibrium in the case, when the two player are producing just one good and the production set have an empty intersection}

Let us recall that the properties of the modulus of convexity $\delta_{\|\cdot\|}$ are investigated if the Banach space is at least two dimensional.
As far as $\mathbb{R}$, endowed with its canonical norm is a subspace of $\mathbb{R}^2_2$ we have
$\delta_{(\mathbb{R},|\cdot|)}(\varepsilon)\geq \delta_{(\mathbb{R}^2_2,\|\cdot\|_2)}(\varepsilon)=\frac{\varepsilon^2}{8}$.
A direct calculaion shows that $\delta_{(\mathbb{R},|\cdot|)}(\varepsilon)=\frac{\varepsilon}{2}$ \cite{Ilchev-zlatanov}.

We will formulate Assumption \ref{assumption3} in the case when the underlying Banach space is $(\mathbb{R},|\cdot|)$.

\begin{assumption}\label{assumption4}
	\begin{enumerate}[\normalfont (1)] Let there is a duopoly market, satisfying the following assumptions:
		
		\item The two firms are producing homogeneous perfect substitute products
		
		\item The first firm can produce qualities from the set $A_x$ and the second firm can produce qualities from the set $A_y$, where $A_x$ and $A_y$ be nonempty closed intervals of $(\mathbb{R},|\cdot|)$
		
		\item Let there exist a close and convex subset $D\subseteq A_x\times A_y$ and maps $F:D\to A_x$ and $f:D\to A_y$, such that $(F(x,y),f(x,y))\subseteq D$ for every $(x,y)\in D$, be the response functions for firm one and two respectively
		
		\item Let there exist $\alpha,\beta>0$, $\alpha+\beta<1$, such that
		\begin{equation}\label{equation:16}
		|F(x,y)-f(u,v)|\leq \alpha |x-v|+\beta |y-u|+(1-(\alpha+\beta))d
		\end{equation}
		for all $(x,y),(u,v)\in A_x\times A_y$, where $d={\rm dist}(A_x,A_y)=\inf\{|x-y|:x\in A_x,y\in A_y\}$.
	\end{enumerate}
\end{assumption}

Then there exists a unique pair $(\xi,\eta)$ in $A_x\times A_y$, which is a coupled best point for the pair of maps $(F,f)$,
i.e. a market equilibrium pair.
Moreover the iteration sequences $\{x_{n}\}_{n=0}^\infty$ and $\{y_n\}_{n=0}^\infty$+9 converge to $\xi$ and $\eta$, respectively.

\begin{enumerate}[\normalfont (i)]
	\item\label{item_ii*} a priori error estimates hold
	$$
	\left|\xi-x_{m}\right|
	\leq 2\max\{|x_0-y_0|,|x_0-y_1|\}\frac{\max\{W_{0,1}(x,y),W_{0,0}(x,y)\}}{d}\cdot\frac{(\alpha+\beta)^{m}}{1-(\alpha+\beta)};
	$$
	$$
	\left|\eta-y_{m}\right|
	\leq 2\max\{|x_0-y_0|,|x_1-y_0|\}\frac{\max\{W_{0,1}(y,x),W_{0,0}(y,x)\}}{d}\cdot\frac{(\alpha+\beta)^{m}}{1-(\alpha+\beta)};
	$$
	\item\label{item_iii-1} a posteriori error estimates hold
	$$
	\left|\xi-x_{n}\right|
	\leq 2\max\{|x_{n-1}-y_{n-1}|,|x_{n-1}-y_{n}|\}\frac{\max\{W_{n-1,n}(x,y),W_{n-1,n-1}(x,y)\}}{d}
	\frac{\alpha+\beta}{1-(\alpha+\beta)}.
	$$
	$$
	\left|\eta-y_{n}\right|
	\leq 2\max\{|x_{n-1}-y_{n-1}|,|x_{n}-y_{n-1}|\}\frac{\max\{W_{n-1,n}(y,x),W_{n-1,n-1}(y,x)\}}{d}
	\frac{\alpha+\beta}{1-(\alpha+\beta)},
	$$
	where $W_{n,m}(x,y)=|x_n-x_m|-d$.
\end{enumerate}

The proof is a direct consequence of Theorem \ref{th:main} and the remark that $(\mathbb{R},|\cdot|)$ is an uniformly convex Banach space with modulus of convexity $\delta_{|\cdot|}(\varepsilon)=\frac{\varepsilon}{2}$.

\subsubsection{Example when the two players are producing just one good}

Let us consider a market with two competing firms, producing two products, that are perfect substitutes. Let us assume that the first firm can produce
much smaller quantities than the second one, i.e. $x, y$, so that $x\in [0,1]$ and $y\in [2,3]$. 
Let us consider the response functions $F(x,y)$ and $f(x,y)$ defined by
$$
F(x,y)=\frac{x}{2}-\frac{y}{4}+1,
\ \
f(x,y)=-\frac{u}{4}+\frac{v}{2}+\frac{5}{4}
$$
It is easy to see that $F:[0,1]\times [2,3]\to [0,1]$ and $f:[0,1]\times [2,3]\to [2,3]$

Indeed the inequalities $0\leq\frac{x}{2}-\frac{y}{4}+1\leq 1$ are equivalent to
$$
\left|
\begin{array}{rcl}
\frac{y}{4}&\leq& \frac{x}{2}+1\\[10pt]
\frac{x}{2}&\leq& \frac{y}{4}
\end{array}
\right.
$$
for $(x,y)\in [0,1]\times [2,3]$.

The inequalities $2\leq-\frac{u}{4}+\frac{v}{2}+\frac{5}{4}\leq 3$ are equivalent to
$$
\left|
\begin{array}{rcl}
\frac{3}{4}+\frac{u}{4}&\leq&\frac{v}{2}\\[10pt]
\frac{v}{2}&\leq& \frac{7}{4}+\frac{u}{4}
\end{array}
\right.
$$
for $(u,v)\in [0,1]\times [2,3]$.  

Then we obtain
$$
\begin{array}{lll}
|F(x,y)-f(u,v)|&=&\left|-\displaystyle\frac{u}{4}+\frac{v}{2}+\displaystyle\frac{5}{4}-\left(\frac{x}{2}-\frac{y}{4}+1\right)\right|
\leq\displaystyle\frac{|v-x|}{2}+\frac{|y-u|}{4}+\left|\frac{5}{4}-1\right|\\[10pt]
&=&\displaystyle\frac{|v-x|}{2}+\frac{|y-u|}{4}+\frac{1}{4}
=\displaystyle\frac{|v-x|}{2}+\frac{|y-u|}{4}+\left(1-\left(\frac{1}{2}+\frac{1}{4}\right)\right)d.
\end{array}
$$
Therefore
the ordered pair $(F,f)$ satisfies Assumption \ref{assumption4} with constants $\alpha=\frac{1}{2}$, $\beta=\frac{1}{4}$.
Thus there exists an equilibrium pair $(x,y)$ and for any initial start in the economy the iterated sequences $(x_n,y_n)$ converge to the 
market equilibrium $(x,y)$. We get in this case that the equilibrium pair of the production of the two firms is $x=1$, $y=2$ and the total production will be $a=3$.

\begin{center}
	\captionof{table}{Values of the iterated sequence $(x_n,y_n)$ if stared with $(0.2,2.8)$}
	\label{table18}
	\begin{tabular*}{0.9\textwidth}{@{\extracolsep{\fill} }lrrrrrrr}
		\hline
		$n$&0&1&2&5&10&20&30\\
		\hline
		$x_n$&0.2&0.4&0.55&0.81&0.95&0.997&0.9998\\
		\hline
		$y_n$&2.8&2.6&2.45&2.18&2.04&2.002&2.0001\\
		\hline
	\end{tabular*}
\end{center}

\begin{center}
	\captionof{table}{Number $n$ of iterations needed by the a priori estimate if stared with $(0.2,2.8)$}
	\label{table19}
	\begin{tabular*}{0.9\textwidth}{@{\extracolsep{\fill} }lrrrrr}
		\hline
		$\varepsilon$&0.1&0.01&0.001&0.0001&0.00001\\
		\hline
		$n$&21&29&37&45&53\\
		\hline
	\end{tabular*}
\end{center}

\begin{center}
	\captionof{table}{Number $n$ of iterations needed by the a posteriori estimate if stared with $(100,20)$}
	\label{table20}
	\begin{tabular*}{0.9\textwidth}{@{\extracolsep{\fill} }lrrrrr}
		\hline
		$\varepsilon$&0.1&0.01&0.001&0.0001&0.00001\\
		\hline
		$n$&17&25&33&41&49\\
		\hline
	\end{tabular*}
\end{center}

\subsection{Conclusion}

Markets dominated by a small group of players are not uncommon even in a fast moving global economy. Therefore it is essential to analyze these cases, understand what leads to equilibrium and how different companies respond to changes in the economic environment. In this paper we have built a model on existence and uniqueness of market equilibrium in oligopoly markets that can be derived from response functions of major players. We assume that goods produced by different market players are perfect substitutes as this simplifies mathematical description of the model. Due to the fact that response functions can also account for difference in product qualities, the model can also be applied to situations where price is not the only factor on which companies compete. With carefully constructed response function it is possible to fully comprehend all five basic factors influencing competition - product features, number of sellers, information ability and barriers to entry.
Existence and uniqueness of equilibrium can be analyzed with both linear and non-linear response functions which proves to be a very flexible approach when studying different markets, which despite being dominated by small number of companies may have quite different characteristics. Applications of the suggested model when production sets of market participants have empty intersection are particularly important when it is necessary to account for real world limitations like huge economies of scale or unique resources available to some players.
One specific application and advantage of the suggested model is that calibration of response functions can be performed in a way that matches observed past behavior (output and prices) or major market players. This way it is possible to assess not only equilibrium stability but also the way that different companies react to changes in the environment.

\bibliographystyle{plain}
\bibliography{Oligopoly}

\begin{thebibliography}{10}
\expandafter\ifx\csname url\endcsname\relax
  \def\url#1{\texttt{#1}}\fi
\expandafter\ifx\csname urlprefix\endcsname\relax\def\urlprefix{URL }\fi
\expandafter\ifx\csname href\endcsname\relax
  \def\href#1#2{#2} \def\path#1{#1}\fi

\bibitem{EV}
A.~Eldred, P.~Veeramani, Existence and convergence of best proximity points, J.
  Math. Anal. Appl. 323~(2) (2006) 1001--1006.

\bibitem{GL}
D.~Guo, V.~Lakshmikantham, Coupled fixed points of nonlinear operators with
  applications, Nonlinear Anal. 11~(5) (1987) 623--632.

\bibitem{GRK}
A.~Gupta, S.~Rajput, P.~Kaurav, Coupled best proximity point theorem in metric
  spaces, International Journal of Analysis and Applications 4~(2) (2014)
  201--215.

\bibitem{SK}
W.~Sintunavarat, P.~Kumam, Coupled best proximity point theorem in metric
  spaces, Fixed Point Theory Appl. 2012/1/93.

\bibitem{Berinde22}
V.~Berinde, Generalized coupled fixed point theorems for mixed monotone
  mappings in partially ordered metric spaces, Nonlinear Anal., 74~(18) (2011)
  7347--7355.

\bibitem{Berinde23}
V.~Berinde, Coupled fixed point theorems for $\phi$--contractive mixed monotone
  mappings in partially ordered metric spaces, Nonlinear Anal., 75~(6) (2012)
  3218--3228.

\bibitem{BP}
V.~Berinde, M.~P\u{a}curar, A constructive approach to coupled fixed point
  theorems in metric spaces, Carpathian J. Math., 31~(3) (2015) 277--287.

\bibitem{Cournot}
A.~A. Cournot, Researches Into the Mathematical Principles of the Theory of
  Wealth, translation Edition, Macmillan, New York, 1897.

\bibitem{Friedman}
J.~W. Friedman, Oligopoly Theory, {C}ambradge {U}niversity {P}ress, 2007.

\bibitem{Smith}
A.~Smith, A Mathematical Introduction to Economica, Basil Blackwell Limited,
  1987.

\bibitem{Ber}
V.~Berinde, Iterative Approximation of Fixed Points, Springer, Berlin, 2007.

\bibitem{Z}
B.~Zlatanov, Error estimates for approximating of best proximity points for
  cyclic contractive maps, Carpathian J. Math. 32~(2) (2016) 241--246.

\bibitem{Ilchev-zlatanov}
A.~Ilchev, B.~Zlatanov, Error estimates for approximation of coupled best
  proximity points for cyclic contractive maps, Applied Mathematics and
  Computation 290 (2016) 412--425.

\bibitem{Zlatanov-FPT}
B.~Zlatanov, Coupled best proximity points for cyclic contractive maps and
  their applications, Fixed Point Theory 22 (2021) to appear.

\bibitem{M}
A.~Meir, On the uniform convexity of $l_p$ spaces, $1<p<2$s, Illinois J. Math.
  28~(3) (1984) 420--424.

\bibitem{BB}
B.~Beauzamy, Introduction to Banach Spaces and their Geometry, North--Holland
  Publishing Company, Amsterdam, 1979.

\bibitem{DGZ}
R.~Deville, G.~Godefroy, V.~Zizler, Smothness and renormings in Banach spaces,
  Pitman Monographs and Surveys in Pure and Applied Mathematics, 1993.

\bibitem{FHHMPZ}
M.~Fabian, P.~Habala, P.~H\'{a}jek, V.~Montesinos, J.~Pelant, V.~Zizler,
  Functional Analysis and Infinite-Dimensional Geometry, Springer New York,
  2011.

\bibitem{Ueda}
M.~Ueda, Effect of information asymmetry in {C}ournot duopoly game with bounded
  rationality, Applied Mathematicsand Computation 362 (2019) Article number
  124535.

\bibitem{Rubinstein}
A.~Rubinstein, Modeling Bounded Rationality, MIT Press, Massachusetts, 2019.

\bibitem{Cellini-Lambertini}
R.~Cellini, L.~Lambertini, Dynamic oligopoly with sticky prices: Closed-loop,
  feedback and open-loop solutions, J. Dyn. Control Syst. 10 (2004) 303--314.

\bibitem{BhanuSinha}
U.~B. Sinha, Optimal value of a patent in an asymmetric {C}ournot duopoly
  market, Review of Economic Studies 57 (2016) 93--105.

\bibitem{ZMM}
M.~A. Zapata, L.~M. Caraballo1, A.~M. Marmol, Hurwicz's criterion and the
  equilibria of duopoly models, Central European Journal of Operations Research
  27~(4) (2019) 937--952.

\bibitem{Anderson-Engers}
S.~P. Anderson, M.~Engers, Stackelberg versus {C}ournot oligopoly equilibrium,
  International Journal of Industrial Organization 10~(1) (1992) 127--135.

\bibitem{BarCoj}
A.~Barbagallo, M.-G. Cojocaru, Dynamic equilibrium formulation of the
  oligopolistic market problem, Mathematical and Computer Modelling 49~(4--5)
  (2009) 966--976.

\bibitem{ChZhW}
C.~K. Chan, Y.~Zhou, K.~H. Wong, A dynamic equilibrium model of the
  oligopolistic closed-loop supply chain network under uncertain and
  time-dependent demands, Transportation Research Part E: Logistics and
  Transportation Review 118~(C) (2018) 325--354.

\bibitem{Klemm}
A.~Klemm, Profit maximisation and alternatives in oligopolies, Economics
  Working Paper Archive at WUST Industrial Organization, University Library of
  Munich (2004) https://EconPapers.repec.org/RePEc:wpa:wuwpio:0409003.

\bibitem{Teocharis}
R.~D. Teocharis, On the stability of the {C}ournot solution of the oligopoly
  problem, Review of Economic Studies 27 (1960) 133--134.

\bibitem{McManus-Quandt}
M.~McManus, R.E.Quandt, Comments on the stability of the {C}ournot oligopoly
  model, Review of Economic Studies 27 (1961) 136--139.

\bibitem{BCKS}
G.~Bischi, C.~Chiarella, M.~Kopel, F.~Szidarovszky, Nonlinear Oligopolies
  Stability and Bifurcations, Springer-Verlag Berlin Heidelberg, 2010.

\bibitem{MS}
A.~Matsumoto, F.~Szidarovszky, Dynamic Oligopolies with Time Delays, Springer
  Nature Singapore Pte Ltd., 2018.

\bibitem{Okuguchi}
K.~Okuguchi, Expectations and Stability in Oligopoly Models, Springer Nature
  Singapore Pte Ltd., 1976.

\bibitem{OS}
K.~Okuguchi, F.~Szidarovszky, The Theory of Oligopoly with Multi--Product
  Firms, Springer-Verlag Berlin Heidelberg, 1990.

\end{thebibliography}

\end{document}